\documentclass[12pt]{article}

\usepackage{amsmath, amsthm, amssymb, stmaryrd}
\usepackage[all]{xy}
\usepackage{fullpage}
\usepackage{titlesec}
\usepackage{mathrsfs}
\usepackage{amsbsy}
\usepackage{turnstile}
\usepackage{bbm}
\usepackage{yhmath}
\usepackage{tensor}
\usepackage{graphics}
\usepackage{enumitem}
\usepackage{calligra}
\usepackage{verbatim}
\usepackage{setspace}
\usepackage{hhline}
\usepackage{array}

\usepackage[pdftex,bookmarks=true]{hyperref}
\usepackage[svgnames]{xcolor}
\hypersetup{
    linkbordercolor = {PaleTurquoise},
    citebordercolor = {PaleGreen},
}

\usepackage{titletoc}

\titlecontents{section}
[0pt]                                               
{}%
{\contentsmargin{0pt}                               
    \thecontentslabel\enspace%
    }
{\contentsmargin{0pt}}                        
{\titlerule*[.5pc]{.}\contentspage}                 
[]                                                  

\makeatletter
\newcommand*{\@old@slash}{}\let\@old@slash\slash
\def\slash{\relax\ifmmode\delimiter"502F30E\mathopen{}\else\@old@slash\fi}
\makeatother

\titleformat{\section}{\normalsize\bfseries}{\thesection}{1em}{}
\titleformat{\subsection}{\normalsize\bfseries}{\thesubsection}{1em}{}

\numberwithin{equation}{subsection}

\theoremstyle{plain}

\newtheorem{PropSub}[subsection]{Proposition}
\newtheorem{LemSub}[subsection]{Lemma}
\newtheorem{CorSub}[subsection]{Corollary}
\newtheorem{ThmSub}[subsection]{Theorem}

\newtheorem{PropSubSub}[subsubsection]{Proposition}

\theoremstyle{definition}

\newtheorem{DefSub}[subsection]{Definition}
\newtheorem{ExaSub}[subsection]{Example}
\newtheorem{RemSub}[subsection]{Remark}
\newtheorem{ParSub}[subsection]{}

\newtheorem*{Ack}{Acknowledgement}

\newtheorem{DefSubSub}[subsubsection]{Definition}
\newtheorem{RemSubSub}[subsubsection]{Remark}
\newtheorem{ExaSubSub}[subsubsection]{Example}
\newtheorem{ParSubSub}[subsubsection]{}

\newcommand*{\emptybox}{\leavevmode\hbox{}}

\newdir{ >}{{}*!/-10pt/@{>}}

\DeclareMathAlphabet{\mathpzc}{OT1}{pzc}{m}{it}
\DeclareMathAlphabet{\mathcalligra}{T1}{calligra}{m}{n}

\newcommand{\pref}[1]{\textnormal{(\ref{#1})}}

\newcommand{\A}{\ensuremath{\mathscr{A}}}
\newcommand{\B}{\ensuremath{\mathscr{B}}}
\newcommand{\C}{\ensuremath{\mathscr{C}}}
\newcommand{\D}{\ensuremath{\mathscr{D}}}
\newcommand{\E}{\ensuremath{\mathscr{E}}}

\newcommand{\normalJ}{\ensuremath{\mathscr{J}}}
\newcommand{\J}{\ensuremath{{\kern -0.4ex \mathscr{J}}}}

\newcommand{\JF}{\ensuremath{{\kern -0.4ex \mathscr{J}_{\kern -0.05ex F}}}}

\newcommand{\T}{\ensuremath{\mathscr{T}}}

\newcommand{\V}{\ensuremath{\mathscr{V}}}

\newcommand{\X}{\ensuremath{\mathscr{X}}}

\newcommand{\VCAT}{\ensuremath{\V\textnormal{-\text{CAT}}}}

\newcommand{\NN}{\ensuremath{\mathbb{N}}}

\newcommand{\RR}{\ensuremath{\mathbb{R}}}

\newcommand{\TT}{\ensuremath{\mathbb{T}}}

\newcommand{\ob}{\ensuremath{\operatorname{\textnormal{\textsf{ob}}}}}

\newcommand{\ca}[1]{\scalebox{0.85}{\raisebox{0.3ex}{$|$}} #1\scalebox{0.85}{\raisebox{0.3ex}{$|$}}}

\newcommand{\Th}{\ensuremath{\textnormal{Th}}}

\newcommand{\ThJ}{\ensuremath{\Th_{\kern -0.5ex \normalJ}}}
\newcommand{\SubThJ}{\ensuremath{\textnormal{SubTh}_{\kern -0.5ex \normalJ}}}

\newcommand{\Vfp}{\ensuremath{\V_{\kern -0.5ex fp}}}
\newcommand{\TTperpJ}{\ensuremath{\TT^\perp_{\kern-.1ex\scriptscriptstyle\J}}}
\newcommand{\TTperpJwrt}[1]{\ensuremath{\TT^\perp_{\kern-.1ex\scriptscriptstyle\J#1}}}

\newcommand{\inJ}[1]{#1 \in \normalJ\kern -1.2ex}

\newcommand{\VCATJ}{\VCAT_{\kern -0.7ex \normalJ}}
\newcommand{\PhiJ}{\Phi_{\kern -0.8ex \normalJ}}
\newcommand{\MndJ}{\Mnd_{\kern -0.8ex \normalJ}}

\newcommand{\Set}{\ensuremath{\operatorname{\textnormal{\text{Set}}}}}

\newcommand{\Mf}{\ensuremath{\operatorname{\textnormal{\text{Mf}}}}}

\newcommand{\Mnd}{\ensuremath{\operatorname{\textnormal{\text{Mnd}}}}}

\newcommand{\CMon}{\ensuremath{\operatorname{\textnormal{\text{CMon}}}}}

\newcommand{\Alg}[1]{\ensuremath{#1\kern -.5ex\operatorname{\textnormal{-\text{Alg}}}}}
\newcommand{\CAlg}[1]{\ensuremath{#1\kern -.5ex\operatorname{\textnormal{-\text{CAlg}}}}}
\newcommand{\CinftyRing}{\ensuremath{C^\infty\kern -.5ex\operatorname{\textnormal{-\text{Ring}}}}}

\newcommand{\CRProf}{\ensuremath{\operatorname{\textnormal{\text{CRProf}}}}}
\newcommand{\CRProfJ}{\ensuremath{\CRProf_{\kern -0.6ex \normalJ}}}

\newcommand{\DBun}{\ensuremath{\operatorname{\textnormal{\text{DBun}}}}}

\newcommand{\Adjn}[6]{\xymatrix {#1 \ar@/_0.5pc/[rr]_{#2}^(0.4){#4}^(0.6){#5}^{\top} & & #6 \ar@/_0.5pc/[ll]_{#3}}}
\newcommand{\Equiv}[6]{\xymatrix {#1 \ar@/_0.5pc/[rr]_{#2}^(0.4){#4}^(0.6){#5}^{\sim} & & #6 \ar@/_0.5pc/[ll]_{#3}}}

\newcommand{\lt}{\leqslant}

\newcommand{\pushoutcorner}{\ar@{}[dr]|(.3)\ulcorner}
\newcommand{\pullbackcorner}{\ar@{}[dr]|(.3)\lrcorner}

\newcommand{\pover}[3]{#1^{#2{\scriptscriptstyle/#3}}}

\newcommand{\cmt}[1]{}

\begin{document}

\author{\normalsize  Rory B. B. Lucyshyn-Wright\thanks{The author gratefully acknowledges an AARMS PDF held earlier.  The author thanks Geoffrey Cruttwell for useful discussions.}\let\thefootnote\relax
\\
\small Brandon University, Brandon, Manitoba, Canada}

\title{\large \textbf{On the geometric notion of connection\\ and its expression in tangent categories}}

\date{}

\maketitle

\abstract{
\textit{Tangent categories} provide an axiomatic approach to key structural aspects of differential geometry that exist not only in the classical category of smooth manifolds but also in algebraic geometry, homological algebra, computer science, and combinatorics.  Generalizing the notion of \textit{(linear) connection} on a smooth vector bundle, Cockett and Cruttwell have defined a notion of connection on a differential bundle in an arbitrary tangent category.  Herein, we establish equivalent formulations of this notion of connection that reduce the amount of specified structure required.  Further, one of our equivalent formulations substantially reduces the number of axioms imposed, and others provide useful abstract conceptualizations of connections.  In particular, we show that a connection  on a differential bundle $E$ over $M$ is equivalently given by a single morphism $K$ that induces a suitable decomposition of $TE$ as a biproduct.  We also show that a connection is equivalently given by a vertical connection $K$ for which a certain associated diagram is a limit diagram.
}

\section{Introduction} \label{sec:intro}

To view a particular geometric space as merely a smooth manifold, merely an analytic space, or merely a scheme sometimes entails neglecting further relevant structure that the space may carry, such as Riemannian or Hermitian structure.  In differential geometry, the concept of \textit{affine connection} captures one such notion of further geometric structure, a notion that is substantially more general than Riemannian structure and yet still gives rise to many important derived phenomena such as geodesics, curvature, parallel transport, and holonomy.  

Cockett and Cruttwell \cite{CoCr:Conn} have defined a notion of affine connection on an object $M$ of any suitably structured category $\X$, so that the classical notion is recovered when $\X$ is the category $\Mf$ of smooth manifolds.  Cockett and Cruttwell work within the framework of \textit{tangent categories} \cite{Ros:Tan,CoCr:TCats}, which are categories $\X$ in which each object $M$ is equipped with a \textit{tangent bundle} $TM$, with an associated morphism $p_M:TM \rightarrow M$ and several further structural morphisms $+_M$, $0_M$, $\ell_M$, and $c_M$, satisfying certain axioms.  The category $\Mf$ is one example of a tangent category, in the company of many other examples, including categories from algebraic geometry, computer science, combinatorics, and homological algebra.

Inspired by formulations of connections established by Ehresmann and by Patterson \cite{Pat:Cx}, Cockett and Cruttwell define an \textit{affine connection} on a given object $M$ of a tangent category $\X$ as a pair $(K,H)$ consisting of morphisms
$$K\;:\;T^2M \rightarrow TM\;\;\;\;\;\;\;H\;:\;TM \times_M TM \rightarrow T^2M$$
in $\X$, satisfying certain axioms, where $T^2M = TTM$ and we write $TM \times_M TM$ to denote the fibre product of $p_M:TM \rightarrow M$ with itself.

The notion of affine connection on a smooth manifold is subsumed by the more general notion of \textit{(linear) connection} on a vector bundle $q:E \rightarrow M$ over $M$, since an affine connection on $M$ is, by definition, a connection on the tangent bundle $p_M:TM \rightarrow M$ of $M$.  In an arbitrary tangent category, one has available a generalization of the notion of vector bundle that is formulated without reference to the ring of real numbers, namely the notion of \textit{differential bundle} \cite{CoCr:DBun}.  Explicitly, a differential bundle
$$\underline{E} = (E,q,\sigma,\zeta,\lambda)$$
over an object $M$ of $\X$ consists of a morphism $q:E \rightarrow M$ in $\X$ equipped with the structure of a commutative monoid $(E,q,\sigma,\zeta)$ in the slice category $\X \slash M$, together with a specified morphism $\lambda:E \rightarrow TE$ called the \textit{lift}, satisfying certain axioms.  In particular, the tangent bundle $p_M:TM \rightarrow M$ carries the structure of a differential bundle, which we write herein as
$$\underline{T}M = (TM,p_M,+_M,0_M,\ell_M),$$
where $+_M$, $0_M$, $\ell_M$ are certain of the structural morphisms possessed by the tangent category $\X$.

Cockett and Cruttwell define a \textit{connection} \cite{CoCr:Conn} on a differential bundle $\underline{E}$ over $M$ as a pair $(K,H)$ consisting of morphisms
$$K\;:\;TE \rightarrow E\;\;\;\;\;\;\;H\;:\;E \times_M TM \rightarrow TE$$
in $\X$, satisfying several axioms; see \ref{def:cx} below\footnote{As in \cite{CoCr:Conn}, we shall only consider connections on differential bundles $\underline{E}$ that satisfy the following blanket assumption:  For all natural numbers $m,n \in \NN$ the fibre product of $m$ instances of $E$ and $n$ instances of $TM$, over the base object $M$, exists in $\X$ and is preserved by $T^k$ for each $k \in \NN$ \pref{par:std_assn}.}.  A connection in this sense consists of a \textit{vertical connection} $K$ and a \textit{horizontal connection} $H$ that are suitably compatible.  The authors establish certain conditions under which a horizontal connection $H$ determines a compatible vertical connection $K$, and vice versa, except that the passage from a vertical connection $K$ to a compatible horizontal connection $H$ is achieved only under the assumption that \textit{there exists some} (not necessarily compatible) horizontal connection on $\underline{E}$ \cite[Propositions 5.12, 5.13]{CoCr:Conn}.

\medskip

In the present paper, we establish equivalent formulations of the above notion of connection that reduce the amount of specified structure required.  Further, one of our equivalent formulations substantially reduces the number of axioms imposed, and others provide useful abstract conceptualizations of connections, in terms of the standard notion of \textit{biproduct}.  We show that the original axioms for connections in tangent categories arise as subtle generic features of certain biproducts of differential bundles, and this analysis enables us to remove some of these axioms through a new formulation in terms of just a single morphism $K$.  The fact that we have been able to establish these useful equivalent formulations is a further testament to the coherence and suitability of the original definition of Cockett and Cruttwell.

The key idea that enables our approach to connections is the fact that a connection $(K,H)$ on a differential bundle $\underline{E}$ over $M$ induces a `coordinatization' of $TE$, by providing a decomposition of $TE$ as a fibre product
\begin{equation}\label{eq:fp_dec}TE = E \times_M TM \times_M E\end{equation}
over $M$.  Indeed, Cockett and Cruttwell have proved in \cite[Prop. 5.8]{CoCr:Conn} that the diagram
\begin{equation}\label{eq:cx_ld_i}
\xymatrix{
            & TE \ar[dl]_{p_E} \ar[d]|{T(q)} \ar[dr]^K & \\
E \ar[dr]_q & TM \ar[d]|{p_M}                          & E \ar[dl]^q\\
   & M                                                 &
}
\end{equation}
presents $TE$ as such a fibre product.  In particular, an affine connection $(K,H)$ allows us to express the second tangent bundle $T^2M$ as fibre product $T^2M = TM \times_M TM \times_M TM$ over $M$ of three instances of $TM$.  In this way, an affine connection allows second-order tangent vectors to be expressed as triples of tangent vectors at a common point.  In the case where $\X$ is the category $\Mf$ of smooth manifolds, this phenomenon was discovered by Dombrowski \cite{Dom}, who defined the map $K$ in terms of a given Koszul connection $\nabla$.  As a simple example in $\Mf$, one may take $M = \RR^n$, so that $M$ carries a canonical affine connection that allows us to express $T^2M$ as a three-fold fibre product $TM \times_M TM \times_M TM = M \times M^3 = M^4$ of $TM = M^2$ over $M$.

Herein we show that a connection on a differential bundle $\underline{E} = (E,q,\sigma,\zeta,\lambda)$ over $M$ can be described equivalently as a morphism $K$ inducing a fibre product decomposition \eqref{eq:fp_dec} that satisfies certain natural further conditions.  One of our results of this type is as follows:
$$
\begin{minipage}{5.0in}
\textbf{Theorem \ref{thm:main}(2).}  \textit{A connection on $\underline{E}$ is equivalently given by a retraction $K:TE \rightarrow E$ of $\lambda$ such that $TE$ underlies a product of differential bundles $\underline{E} \times_M \underline{T}M \times_M \underline{E}$ over $M$ whose third projection is $K$ and whose first and second partial bundles \pref{def:jth_pb} are
$$\underline{T}E = (TE,p_E,+_E,0_E,\ell_E)\;\;\;\;\text{and}\;\;\;\;T\underline{E} = (TE,T(q),T(\sigma),T(\zeta),T(\lambda)),$$ 
respectively.}
\end{minipage}
$$
A product $\underline{E} \times_M \underline{T}M \times_M \underline{E}$ satisfying the specified conditions is not only unique up to isomorphism if it exists, but moreover its underlying fibre product diagram in $\X$ must be precisely \eqref{eq:cx_ld_i}.

Here we have employed the notion of \textit{$j$-th partial bundle}, which we now explain.  Given a finite product of differential bundles $\underline{F} = \underline{F_1} \times_M ... \times_M \underline{F_n}$ over $M$, if certain products of differential bundles exist, then the underlying object $F$ acquires the structure of a differential bundle over $F_j$, for each $j \in \{1,...,n\}$, which we call the $j$-th \textit{partial bundle} of the given product \pref{def:jth_pb}.  For example, the \textit{first partial bundle} of the given product is the pullback of $\underline{F_2} \times_M ... \times_M \underline{F_n}$ along $F_1 \rightarrow M$.

The category $\DBun_M$ of differential bundles over an object $M$ is an \textit{additive category}, in the sense that it is enriched in commutative monoids, so finite products in $\DBun_M$ are biproducts, which we also call \textit{Whitney sums} with a nod to the classical case.  It follows that we can reformulate the preceding theorem as follows:
$$
\begin{minipage}{5.0in}
\textbf{Theorem \ref{thm:main}(1).}  \textit{A connection on $\underline{E}$ is equivalently given by a morphism $K:TE \rightarrow E$ such that $TE$ underlies a biproduct of differential bundles $\underline{E} \oplus_M \underline{T}M \oplus_M \underline{E}$ over $M$ whose third projection is $K$, whose third injection is $\lambda$, and whose first and second partial bundles are $\underline{T}E$ and $T\underline{E}$, respectively.}
\end{minipage}
$$
The first and second injections of the biproduct are then necessarily $0_E:E \rightarrow TE$ and $T(\zeta):TM \rightarrow TE$.

This formulation in terms of biproducts gives us a straightforward recipe for recovering the associated horizontal connection $H$:
$$
\begin{minipage}{4.7in}
\textit{Given a morphism $K$ as in Theorem \ref{thm:main}(1), the associated horizontal connection $H:E \times_M TM \rightarrow TE$ is precisely the canonical injection
$$H\;\;:\;\;\underline{E} \oplus_M \underline{T}M \longrightarrow \underline{E} \oplus_M \underline{T}M \oplus_M \underline{E}$$
of the first two factors of the given biproduct.}
\end{minipage}
$$

In the formulation of Cockett and Cruttwell, a morphism $K:TE \rightarrow E$ is called a vertical connection on the differential bundle $\underline{E} = (E,q,\sigma,\zeta,\lambda)$ if it is a retraction of $\lambda$ and satisfies axioms (C.1) and (C.2) in \ref{def:cx} below, while a morphism $H:E \times_M TM \rightarrow TE$ is called a horizontal connection on $\underline{E}$ if it is a section of $\langle p_E, T(q) \rangle$ and satisfies axioms (C.3) and (C.4) of \ref{def:cx}.  A pair $(K,H)$ consisting of a vertical connection and a horizontal connection is then called a connection if it satisfies two further axioms, (C.5) and (C.6) in \ref{def:cx}, which demand that $K$ and $H$ be compatible in a suitable sense.  Collectively, the resulting axioms for a connection assert the equality of ten pairs of morphisms in $\X$ (cf. \cite[Lemmas 3.3, 4.6, Def. 5.2]{CoCr:Conn}).  Cockett and Cruttwell prove that if the tangent category $\X$ \textit{has negatives}, meaning that the commutative monoid structure on $\underline{T}M$ is an abelian group for each object $M$, then for any horizontal connection $H$ on $\underline{E}$ there is an associated vertical connection $K$ such that $(K,H)$ is a connection \cite[Prop. 5.12]{CoCr:Conn}.  Again supposing that $\X$ has negatives, Cockett and Cruttwell also prove a result in the opposite direction, except with an additional hypothesis:  If $K$ is a vertical connection on $\underline{E}$ and there exists a horizontal connection on $\underline{E}$, then we can associate to $K$ a horizontal connection $H$ for which $(K,H)$ is a connection \cite[Prop. 5.13]{CoCr:Conn}.

Without assuming negatives or the existence of a horizontal connection on the given differential bundle $\underline{E}$, we prove herein that if a vertical connection $K$ satisfies the further assumption that \eqref{eq:cx_ld_i} is a fibre product diagram in $\X$, then there is a unique horizontal connection $H$ on $\underline{E}$ such that $(K,H)$ is a connection.  This leads to the following economical formulation of connections:
$$
\begin{minipage}{5.0in}
\textbf{Theorem \ref{thm:main}(3).}  \textit{A connection on $\underline{E}$ is equivalently given by a vertical connection $K:TE \rightarrow E$ on $\underline{E}$ such that \eqref{eq:cx_ld_i} is a fibre product diagram in $\X$.
}
\end{minipage}
$$
Thus we have obtained an equivalent formulation of connections in terms of a single morphism $K$ that is required to satisfy the equational axioms for a vertical connection \pref{def:cx} (representing a total of four equations) together with a single `exactness' axiom.  We call a vertical connection $K$ satisfying this additional axiom an \textit{effective vertical connection}.  It follows that a vertical connection $K$ is effective if and only if there exists a (necessarily unique) horizontal connection $H$ such that $(K,H)$ is a connection \pref{thm:eq_cond_k}.

Looking toward future work in tangent categories, the axiomatic simplicity and conceptual insights afforded by the above formulations of connections hold the potential to facilitate work with connections and guide the exploration of notions of \textit{higher-order connection}\footnote{In the context of synthetic differential geometry, notions of higher connection have been introduced in \cite{LaNi,Kock:HiCx}.} that suitably decompose the higher tangent bundles.  In fact, Theorem \ref{thm:main}(3) is applied in \cite{BCL} to facilitate the study of categories of \textit{geometric spaces} associated to tangent categories.

\begin{Ack}
The author thanks the anonymous referee for helpful suggestions, including the idea of introducing a term for limits preserved by the iterates of the tangent functor \pref{def:tangent_limit}.
\end{Ack}

\section{Background}

\subsection{Basic categorical notions and notations}

\begin{ParSubSub}\label{par:cat_notn}
With regard to basic categorical matters, we shall mostly adhere to the notational conventions of \cite{CoCr:TCats}, which are very efficient for some kinds of equational reasoning in tangent categories but require a mild departure from common notations for functor application and whiskering, as follows.  Given morphisms $f:A \rightarrow B$ and $g:B \rightarrow C$ in a category $\C$, we denote the associated composite morphism $A \rightarrow C$ by $fg$, thus employing diagrammatic composition order unless otherwise indicated.  Nevertheless, the result of applying a functor $H$ to a morphism $g$ is written as $H(g)$.  Correspondingly, we will write the composite of functors $G:\B \rightarrow \C$ and $H:\C \rightarrow \D$ in non-diagrammatic order as $HG$ (even though this is not the practice in \cite{CoCr:TCats}).  Given a natural transformation $\gamma:G \Rightarrow G':\B \rightarrow \C$ and functors $F:\A \rightarrow \B$ and $H:\C \rightarrow \D$, we denote the resulting `whiskered' natural transformations by $H(\gamma):HG \Rightarrow HG'$ and $\gamma_F:GF \Rightarrow GF'$.  Given a product of a sequence of $n$ objects in a given category, where $n$ is a natural number, we denote the product projections by $\pi_1,...,\pi_n$ so that our notation accords directly with locutions such as ``the first projection'' (cf. \cite[\linebreak p. 334]{CoCr:TCats}).

All given categories herein are assumed locally small.
\end{ParSubSub}

\begin{ParSubSub}\label{par:fi_pr}
Given an object $C$ of a category $\C$, we shall denote by $\C\slash C$ the \textit{slice category} over $\C$, whose objects we shall write as pairs $(B,f)$ consisting of an object $B$ of $\C$ equipped with a morphism $f:B \rightarrow C$ in $\C$.  We will sometimes omit explicit mention of the associated morphism $f$.  Given a family of morphisms $(f_i:B_i \rightarrow C)_{i \in I}$ in $\C$, a \textbf{fibre product (over $C$)} of the family $(f_i)_{i \in I}$ is, by definition, a product in $\C \slash C$ of the family of objects $(B_i,f_i)$ $(i \in I)$.  We shall denote such a fibre product by $\prod_{i \in I}^M (B_i,f_i)$, or $\prod_{i \in I}^M B_i$ when the associated morphisms $f_i$ are understood.  When $I = \{1,2,...n\}$ for some $n$ natural number $n \in \NN$, we will often denote such a fibre product by $B_1 \times_M B_2 \times_M ... \times_M B_n$.

A fibre product of $(f_i)_{i \in I}$ in $\C$ is equivalently described as a limit, in $\C$, of the diagram consisting of the object $C$ together with the morphisms $f_i:B_i \rightarrow C$ $(i \in I)$.  In the case where $I$ is empty, a cone for this diagram is given by just a single object of $\C \slash C$.  On the other hand, if $I$ is nonempty, then a cone for this diagram is equivalently given by an object $A$ of $\C$ equipped with a family of morphisms $(g_i:A \rightarrow B_i)_{i \in I}$ with the property that $g_if_i = g_jf_j$ for all $i,j \in I$.  Still supposing that $I$ is nonempty, a fibre product $\prod_{i \in I}^M B_i$ over $C$ is equipped with a universal such family $(\pi_i:\prod_{i \in I}^M B_i \rightarrow B_i)_{i \in I}$, which is therefore a jointly monic family of morphisms \textit{in $\C$}.  We shall call the diagram in $\C$ consisting of the morphisms $\pi_i$ and $f_i$ $(i \in I)$ a \textbf{fibre product diagram}.

Given a single morphism $f:B \rightarrow C$ in $\C$ and a natural number $n$, an $n$-th \textbf{fibre power} of $f$ in $\C$ is, by definition, a fibre product over $C$ of the constant family $(f:B \rightarrow C)_{i \in \NN,\;1 \lt i \lt n}$.  The $n$-th fibre power of $f$ shall be denoted by $f^{(n)}:\pover{B}{n}{C} \rightarrow C$.
\end{ParSubSub}

\subsection{Additive categories, commutative monoids, and biproducts}

\begin{ParSubSub}\label{par:add_cat}
By an \textbf{additive category} we shall mean a locally small category $\A$ in which each hom-set $\A(A,B)$ with $A,B \in \ob\A$ is equipped with the structure of a commutative monoid, with operations written as $+$ and $0$, such that for every morphism $f:A \rightarrow B$ and every object $C$, the mappings $\A(C,f):\A(C,A) \rightarrow \A(C,B)$ and $\A(f,C):\A(B,C) \rightarrow \A(A,C)$ are homomorphisms of commutative monoids.  Additive categories in this sense are equivalently described as categories enriched in the symmetric monoidal closed category of commutative monoids.  Our usage of the term \textit{additive category} is non-standard, as the term is more often used to refer to categories that are enriched in abelian groups and have finite biproducts.  We shall say that a functor $F:\A \rightarrow \B$ between additive categories is an \textbf{additive functor} if each mapping $F_{AB}:\A(A,B) \rightarrow \B(FA,FB)$ is a homomorphism of commutative monoids, where $A,B \in \ob\A$.

A \textbf{(finite) biproduct} of a given finite family of objects $(A_i)_{i \in I}$ in an additive category $\A$ is, by definition, an object $A = \oplus_{i \in I}A_i$ of $\A$ equipped with morphisms $\pi_i:A \rightarrow A_i$ and $\iota_i:A_i \rightarrow A$ $(i \in I)$ such that (1) $\iota_i\pi_i = 1_{A_i}:A_i \rightarrow A_i$ for all $i \in I$, (2) $\iota_i\pi_j = 0:A_i \rightarrow A_j$ for all $i,j \in I$ with $i \neq j$, and (3) $\sum_{i \in I} \pi_i\iota_i = 1_A:A \rightarrow A$, where the latter finite sum is taken in the commutative monoid $\A(A,A)$.  We call the morphisms $\pi_i$ and $\iota_i$ \textbf{projections} and \textbf{injections}, respectively.  Clearly any additive functor $F:\A \rightarrow \B$ preserves finite biproducts, in the evident sense.

Given a finite biproduct $A = \oplus_{i \in I}A_i$ in an additive category $\A$, the associated morphisms $\pi_i$ necessarily present $A$ as a product $\Pi_{i \in I}A_i$ in $\A$, and the morphisms $\iota_i$ present $A$ as a coproduct\footnote{In fact, the notion of biproduct is sometimes defined in terms of specified product and coproduct structure satisfying conditions (1) and (2) of the preceding paragraph.} $\amalg_{i \in I}A_i$ in $\A$.

Conversely, any finite product $\Pi_{i \in I}A_i$ in an additive category $\A$ necessarily carries the structure of a biproduct $\oplus_{i \in I}A_i$.  Indeed, the needed morphisms $\iota_i:A_i \rightarrow A$ are uniquely defined by conditions (1) and (2) in the above definition of biproduct.

As a consequence, every additive functor preserves finite products.
\end{ParSubSub}

\begin{ParSubSub}\label{par:cmonc}
Given an arbitrary category $\C$, we denote by $\CMon(\C)$ the \textbf{category of commutative monoids in} $\C$.  Explicitly, an object $C$ of $\CMon(\C)$ is given by an object $\ca{C}$ of $\C$ equipped with \textit{specified} finite powers $\ca{C}^n$ $(n \in \NN)$ and morphisms $0:1 \rightarrow \ca{C}$ and $+:\ca{C}^2 \rightarrow \ca{C}$ satisfying the usual diagrammatic associativity, commutativity, and unit axioms; a morphism in $\CMon(\C)$ is given by a morphism in $\C$ that preserves the operations $0$ and $+$.  Hence $\CMon(\C)$ is isomorphic to the category of $\T$-algebras in $\C$ for the Lawvere theory $\T$ of commutative monoids.

Note that $\CMon(\C)$ carries the structure of an additive category \pref{par:add_cat}, with the obvious `pointwise' addition operation on morphisms.  Hence finite products in $\CMon(\C)$ are equivalently described as finite biproducts in $\CMon(\C)$.  Further, finite products in $\CMon(\C)$ are formed as in $\C$:
\end{ParSubSub}

\begin{PropSubSub}\label{thm:bipr_cmonc}
Let $(C_i)_{i \in I}$ be a finite family of commutative monoids in an arbitrary category $\C$, with underlying objects $\ca{C_i}$ in $\C$.  Then a product $\prod_{i \in I}C_i$ in $\CMon(\C)$ is equivalently given by a product $X = \prod_{i \in I}\ca{C_i}$ in $\C$ equipped with specified finite powers $X^n$ in $\C$ $(n \in \NN)$.
\end{PropSubSub}
\begin{proof}
Given a product $X = \prod_{i \in I}\ca{C_i}$ in $\C$ with specified finite powers, it is well-known and easily verified that there are unique morphisms $+:X^2 \rightarrow X$ and $0:X^0 \rightarrow X$ in $\C$ such that $C = (X,+,0)$ is a commutative monoid and the projections $\pi_i:X \rightarrow \ca{C_i}$ are morphisms of commutative monoids $\pi_i:C \rightarrow C_i$.  Conversely, let $C = \prod_{i \in I}C_i$ be a product in $\CMon(\C)$.  Given any object $Z$ of $\C$, the functor $\C(Z,-):\C \rightarrow \Set$ preserves finite products and hence lifts to an additive functor $\C(Z,-)_*:\CMon(\C) \rightarrow \CMon(\Set)$, which necessarily preserves finite products, by \ref{par:add_cat}.  The forgetful functor $\CMon(\Set) \rightarrow \Set$ preserves limits, so the composite functor
$$\CMon(\C) \xrightarrow{\C(Z,-)_*} \CMon(\Set) \longrightarrow \Set$$
preserves finite products.  But this composite functor sends the product cone $(\pi_i:C \rightarrow C_i)_{i \in I}$ to the cone $(\C(Z,\pi_i):\C(Z,\ca{C}) \rightarrow \C(Z,\ca{C_i}))_{i \in I}$, which is therefore a product cone in $\Set
$.  Therefore, the morphisms $\pi_i:\ca{C} \rightarrow \ca{C_i}$ present $\ca{C}$ as a product in $\C$.
\end{proof}

\begin{RemSubSub}\label{par:cmon_pr_fp}
The preceding proposition shows that the forgetful functor $\CMon(\C) \rightarrow \C$ preserves finite products, \textit{even if it does not have a left adjoint and $\C$ does not have finite products}.  By contrast, familiar arguments to the effect that algebraic functors preserve limits do not apply in such circumstances.
\end{RemSubSub}

\begin{ParSubSub}\label{par:add_bun}
Let $\X$ be an arbitrary category.  Given an object $M$ of $\X$, a \textbf{commutative monoid over $M$}, also called an \textbf{additive bundle} over $M$ is, by definition, a commutative monoid $((E,q),+,0)$ in the slice category $\C \slash M$ (so that $q:E \rightarrow M$ in $\X$).  In view of \ref{par:cmonc}, every commutative monoid $(E,q,+,0)$ over $M$ is equipped with an $n$-th fibre power $q^{(n)}:\pover{E}{n}{M} \rightarrow M$ of $q$ for each $n \in \NN$.  Given commutative monoids $\underline{E} = (E,q,+,0)$ and $\underline{F} = (F,r,+,0)$ over $M$ and $N$, respectively, a \textbf{morphism of additive bundles} $(g,f):\underline{E} \rightarrow \underline{F}$ consists of a pair of morphisms $g:E \rightarrow F$ and $f:M \rightarrow N$ in $\X$ with $qf = gr$ such that the diagrams
$$
\xymatrix{
E \ar[r]^{g}        & F          & & E \times_M E \ar[d]_+ \ar[r]^{g \times g} & F \times_N F \ar[d]^+\\
M \ar[u]^0 \ar[r]_f & N \ar[u]_0 & & E \ar[r]_g                                & F
}
$$
commute, where $g \times g = \langle\pi_1g,\pi_2g\rangle$ denotes the morphism induced by $g$ on each factor.
\end{ParSubSub}

\subsection{Tangent categories}

\begin{ParSubSub}\label{par:abs}
Let $\X$ be a category equipped with an endofunctor $T:\X \rightarrow \X$ and a natural transformation $p:T \rightarrow 1_\X$.  Suppose that for each object $M$ of $\X$ and each natural number $n \in \NN$, the morphism $p_M:TM \rightarrow M$ has an $n$-th fibre power 
$$p^{(n)}_M\;\;:\;\;T_nM = \pover{(TM)}{n}{M} \longrightarrow M$$
in $\X$ that is preserved by the $m$-fold composite functor $T^m$ for each $m \in \NN$.  These fibre powers induce a functor $T_n:\X \rightarrow \X$ for each $n \in \NN$.  Given natural transformations $+:T_2 \rightarrow T$, $0:1_\X \rightarrow T$, let us say that $\underline{T} := (T,p,+,0)$ is an \textbf{additive bundle structure} on $\X$ if for each object $M$ of $\X$, $(TM,p_M,+_M,0_M)$ is a commutative monoid over $M$ \pref{par:add_bun}.
\end{ParSubSub}

\begin{ParSubSub}[\textbf{Notation}]\label{par:notn_abs}
Given an object $M$ of a category $\X$ equipped with an additive bundle structure $\underline{T} = (T,p,+,0)$, let
$$\underline{T}M = (TM,p_M,+_M,0_M)$$
denote the associated commutative monoid over $M$.  In particular, $T^2M = TTM$ carries the structure of a commutative monoid over $TM$, namely
$$\underline{T}(TM) = (T^2M,p_{TM},+_{TM},0_{TM})\;.$$
Also, given any commutative monoid $\underline{E} = (E,q,\sigma,\zeta)$ over $M$, whose addition and zero morphisms we write as $\sigma:E \times_M E \rightarrow E$ and $\zeta:M \rightarrow E$, respectively, note that if $T$ preserves finite fibre powers of $q$ then
\begin{equation}\label{eq:te_add}T\underline{E} = (TE,T(q),T(\sigma),T(\zeta))\end{equation}
is a commutative monoid over $TM$ in $\X$.  In particular, this entails that 
$$T(\underline{T}M) = (T^2M,T(p_M),T(+_M),T(0_M))$$
is a commutative monoid over $TM$ in $\X$, so that $T^2M$ thus carries \textit{two} canonical commutative monoid structures over $TM$.
\end{ParSubSub}

\begin{ParSubSub}\label{par:tcats}
By definition, a \textbf{tangent category} \cite{CoCr:TCats} is a category $\X$ equipped with an additive bundle structure $\underline{T} = (T,p,+,0)$ and natural transformations 
$$\ell:T \rightarrow T^2 = TT\;,\;\;\;\;\;\;c:T^2 \rightarrow T^2,$$
called the \textit{vertical lift} and the \textit{canonical flip}, such that
\begin{enumerate}
\item for each object $M$ of $\X$, 
\begin{enumerate}
\item[] $(\ell_M,0_M)\;\;:\;\;\underline{T}M \rightarrow T(\underline{T}M),$
\item[] $(c_M,1_{TM})\;\;:\;\;T(\underline{T}M) \rightarrow \underline{T}(TM)$ 
\end{enumerate}
are morphisms of additive bundles \pref{par:add_bun};
\item the following equations hold (recalling our notations for whiskering \ref{par:cat_notn})
\begin{enumerate}
\item[] $cc = 1\;:\;T^2 \rightarrow T^2,\;\;\;\ell c = \ell\;:\;T \rightarrow T^2,\;\;\;\ell T(\ell) = \ell\ell_T\;:\;T \rightarrow T^3$,
\item[] $T(c)c_TT(c) = c_TT(c)c_T\;:\;T^3 \rightarrow T^3,\;\;\;\ell_TT(c)c_T = cT(\ell)\;:\;T^2 \rightarrow T^3$;
\end{enumerate}
\item the axiom of \textit{universality of the vertical lift} \cite[2.3]{CoCr:TCats} holds.
\end{enumerate}
We say that a tangent category $\X$ \textbf{has negatives} if for every object $M$ of $\X$ the commutative monoid $(TM,p_M,+_M,0_M)$ is an abelian group in $\X \slash M$.
\end{ParSubSub}

\begin{ExaSubSub}[\textbf{Local coordinates for the second tangent bundle}]\label{exa:sman}
The category of smooth manifolds $\Mf$ is a tangent category \cite{CoCr:TCats}.  For each smooth manifold $M$, $p_M:TM \rightarrow M$ is the tangent bundle of $M$, and $T:\Mf \rightarrow \Mf$ assigns to each smooth map $f$ its associated differential $T(f)$.  The structural morphisms $p_M,+_M,0_M,\ell_M,c_M$ admit simple descriptions in terms of local coordinates, by applying the following description on each coordinate patch \cite{CoCr:TCats}.  Letting $V = \RR^n$ be a finite-dimensional vector space, we can express the tangent bundle of $V$ as the first\footnote{\label{fn:fst_pr}In \cite{CoCr:TCats} the \textit{second} projection is employed in such situations, whereas we now employ instead the first projection in order to accord with the common differential-geometric convention of writing the base-point $x$ in the leftmost coordinate and the vector $t$ in the rightmost.} projection $p_V = \pi_1:TV = V^2 = V \times V \rightarrow V$, so that elements of $TV$ are pairs $(x,t)$ consisting of a \textit{point} $x \in V$ and a \textit{tangent vector} $t \in V$, with $p_V(x,t) = x$.  Now $0_V:V \rightarrow TV$ is the map $V \rightarrow V^2$ given by $x \mapsto (x,0)$.  The second fibre power $T_2V \rightarrow V$ of $p_V$ may be expressed as the first projection $\pi_1:V^3 \rightarrow V$, and the fibrewise addition map $+_V:T_2V \rightarrow TV$ is the map $V^3 \rightarrow V^2$ given by $(x,t_1,t_2) \mapsto (x,t_1 + t_2)$.

Since the `total space' $TV = V^2$ of the tangent bundle of $V = \RR^n$ is itself a finite-dimensional vector space, its tangent bundle $p_{TV}:T^2V = TTV \rightarrow TV$ is the first projection\footnote{\label{fn:l_to_r}In \cite{CoCr:TCats}, the second projection is used  (cf. Footnote \ref{fn:fst_pr}) and hence the components of $T^2V = V^4$ are written in the opposite order therein, with the base-point in the rightmost coordinate.} $\pi_1:(V \times V) \times (V \times V) \longrightarrow V \times V$, which (upon omitting parentheses) we may express as
$$p_{TV} = \langle\pi_1,\pi_2\rangle\;\;:\;\;T^2V = V^4 \longrightarrow V^2 = TV,\;\;\;\;(x,t,u,v) \mapsto (x,t)\;.$$
Hence we regard an element $\delta = (x,t,u,v)$ of $T^2V = V^4$ as consisting of a base-point $x$, a tangent vector $t$ at $x$, and a tangent vector $(u,v)$ at $(x,t)$.  Further, we regard $u$ as a tangent vector at $x$, and we regard $v$ as a tangent vector at $t$ (since the differentials $T(\pi_1),T(\pi_2):T(V^2) \rightarrow TV$ send $\delta$ to the elements $(x,u),(t,v) \in TV = V^2$, respectively).  For this reason, we think of $v$ as a \textit{second-order} tangent vector, whereas we regard $t$ and $u$ as \textit{first-order} tangent vectors.

The canonical flip $c_V:T^2V \rightarrow T^2V$ is given by
$$c_V = \langle\pi_1,\pi_3,\pi_2,\pi_4\rangle\;:\;T^2V = V^4 \longrightarrow V^4 = T^2V,\;\;(x,t,u,v) \mapsto (x,u,t,v)$$
The map $T(p_V)\;:\;T^2V \rightarrow TV$ can be expressed as
$$T(p_V) = \langle\pi_1,\pi_3\rangle\;\;:\;\;T^2V = V^4 \longrightarrow V^2 = TV ,\;\;(x,t,u,v) \mapsto (x,u)\;.$$
\end{ExaSubSub}

\begin{RemSubSub}
It is important to note that a tangent category $\X$ is not assumed to have pullbacks.  Indeed, the category of smooth manifolds $\Mf$ \pref{exa:sman} does not have all pullbacks.  However $\Mf$ does have certain pullbacks, including pullbacks of \textit{submersions} along arbitrary smooth maps; a smooth map $f:X \rightarrow Y$ is called a submersion if its differential $T(f)$ is \textit{fibrewise surjective}, in the sense that for any $x \in X$ and any tangent vector $w$ at $f(x)$ on $Y$, there exists a tangent vector $v$ at $x$ on $X$ with $(T(f))(v) = w$.
\end{RemSubSub}

We shall now introduce some convenient terminology for certain classes of limits with which we shall work frequently in this paper:

\begin{DefSubSub}\label{def:tangent_limit}
Given a tangent category $\X$, a \textbf{tangential limit} in $\X$ is a limit in $\X$ that is preserved by $T^n:\X \rightarrow \X$ for each $n \in \NN$.  We shall employ similar terminology with respect to specific kinds of limits, thus defining notions of \textbf{tangential pullback}, \textbf{tangential fibre product}, and $m$-th \textbf{tangential fibre power} $(m \in \NN)$ as special cases of the notion of tangential limit.  For example, in any tangent category, the finite fibre powers of each projection $p_M$ are tangential.
\end{DefSubSub}

\subsection{Differential bundles}

Throughout the remainder of the paper, $\X$ will denote a given tangent category.

\begin{ParSubSub}\label{par:dbun}
Given an object $M$ of $\X$, a \textbf{differential bundle} \cite{CoCr:DBun} over $M$ is a commutative monoid $\underline{E} = (E,q,\sigma,\zeta)$ over $M$ \pref{par:add_bun} equipped with a morphism $\lambda:E \rightarrow TE$ (called the \textit{lift}) such that
\begin{enumerate}
\item for each $n \in \NN$, the $n$-th fibre power of $q$ is tangential \pref{def:tangent_limit};
\item $(\lambda,0_M)\;:\;\underline{E} \rightarrow T\underline{E}$ is a morphism of additive bundles;
\item $(\lambda,\zeta)\;:\;\underline{E} \rightarrow \underline{T}E$ is a morphism of additive bundles;
\item the following diagram is a tangential pullback \pref{def:tangent_limit}
$$
\xymatrix{
E \times_M E \ar[d]_{\text{proj}} \ar[r]^\mu & TE \ar[d]^{T(q)}\\
M \ar[r]_{0_M}                               & TM
}
$$
where $\text{proj} = \pi_1q = \pi_2q$ denotes the projection and $\mu$ is the composite
$$E \times_M E \xrightarrow{\langle\pi_1\lambda,\pi_20_E\rangle} TE \times_{TM} TE = T(E \times_M E) \xrightarrow{T(\sigma)} TE\;;$$
\item $\lambda\ell_E = \lambda T(\lambda)$.
\end{enumerate}
By \ref{par:add_bun}, a differential bundle $(E,q,\sigma,\zeta,\lambda)$ is equipped with a specified fibre power $q^{(n)}:\pover{E}{n}{M} \rightarrow M$ of $q$ for each $n \in \NN$.
\end{ParSubSub}

\begin{ExaSubSub}
For each object $M$ of $\X$, $\underline{T}M = (TM,p_M,0_M,+_M)$ \pref{par:notn_abs} is a differential bundle over $M$ when equipped with the vertical lift $\ell_M:TM \rightarrow T^2M$ \cite[Example 2.4]{CoCr:DBun}.  This differential bundle, which we shall denote also by
$$\underline{T}M = (TM,p_M,0_M,+_M,\ell_M),$$
deserves to be called the \textbf{tangent bundle} of $M$.  The axiom of \textit{universality of the vertical lift} \pref{par:tcats} requires precisely that $\underline{T}M$ satisfy axiom 4 of \ref{par:dbun}; see \cite[Def. 2.1]{CoCr:DBun} and the footnote there. 
\end{ExaSubSub}

\begin{ExaSubSub}\label{exa:vb}
Any smooth vector bundle $q:E \rightarrow M$ carries the structure of a differential bundle in the tangent category $\Mf$ of smooth manifolds; in fact, among differential bundles in $\Mf$, those that are smooth vector bundles in the traditional sense are precisely those whose fibres have constant dimension \cite[Example 2.4]{CoCr:DBun}.
\end{ExaSubSub}

\begin{ExaSubSub}\label{exa:te}
Given any differential bundle $\underline{E} = (E,q,\sigma,\zeta,\lambda)$ over $M$, there is an associated differential bundle
\begin{equation}\label{eq:te}T\underline{E} = (TE,T(q),T(\sigma),T(\zeta),T(\lambda)c_E)\end{equation}
over $TM$, whose lift morphism is the composite $TE \xrightarrow{T(\lambda)} T^2E \xrightarrow{c_E} T^2 E$ \cite[Lemma 2.5]{CoCr:DBun}.  Note that the notation \eqref{eq:te} extends our similar notation for additive bundles \eqref{eq:te_add}.
\end{ExaSubSub}

\begin{ParSubSub}\label{par:cat_dbuns}
Given differential bundles $\underline{E} = (E,q,\sigma,\zeta,\lambda)$ and $\underline{F} = (F,r,\sigma',\zeta',\lambda')$ over $M$ and $N$, respectively, a \textbf{linear morphism of differential bundles} $(g,f):\underline{E} \rightarrow \underline{F}$ consists of a pair of morphisms $g:E \rightarrow F$ and $f:M \rightarrow N$ in $\X$ such that $qf = gr$ and $\lambda T(g) = g\lambda'$. Remarkably, any linear morphism of differential bundles $(g,f):\underline{E} \rightarrow \underline{F}$ is automatically a morphism of additive bundles \cite[Prop. 2.16]{CoCr:DBun}, i.e., it follows that $(g,f)$ preserves the commutative monoid structures.  With these morphisms, differential bundles over various bases are the objects of a category $\DBun$.
\end{ParSubSub}

If $\underline{E}$ and $\underline{F}$ are both differential bundles over $M$, then a \textbf{linear morphism of differential bundles over $M$}, written $g:\underline{E} \rightarrow \underline{F}$, is given by a morphism $g:E \rightarrow  F$ in $\X$ such that $(g,1_M):\underline{E} \rightarrow \underline{F}$ is a linear morphism of differential bundles.  With these morphisms, differential bundles over $M$ are the objects a category $\DBun_M$.

\begin{ParSubSub}\label{eq:ci}
We shall say that differential bundles $\underline{E}$ and $\underline{F}$ over $M$ are \textbf{concretely isomorphic} if there is an isomorphism $\underline{E} \rightarrow \underline{F}$ in $\DBun_M$ that is sent by the forgetful functor $\DBun_M \rightarrow \X \slash M$ to an identity morphism.  Hence $\underline{E}$ and $\underline{F}$ are concretely isomorphic iff their underlying objects $(E,q)$ and $(F,r)$ of $\X \slash M$ are identical and the identity morphism on $E = F$ is an isomorphism $\underline{E} \cong \underline{F}$ in $\DBun_M$.   If this is the case, then $\underline{E}$ and $\underline{F}$ are identical iff they have the same specified $n$-th fibre power of $q = r$ for each $n \in \NN$ (cf. \ref{par:dbun}).  Note that if $\underline{E} = (E,q,\sigma,\zeta,\lambda)$ and $\underline{F} = (F,r,\sigma',\zeta',\lambda')$ are concretely isomorphic, then $E = F$, $q = r$, $\zeta = \zeta'$, and $\lambda = \lambda'$, so $\underline{E}$ and $\underline{F}$ differ only insofar as they may have different specified fibre powers of $q = r$ and addition morphisms $\sigma$, $\sigma'$ that are merely isomorphic (as objects of $\X \slash E$).
\end{ParSubSub}

It is shown in \cite[Lemma 2.7]{CoCr:DBun} that one can `pull back' a differential bundle along a morphism in $\X$, provided that certain assumptions are satisfied.  The precise assumptions needed are expressed within the discussion that precedes Lemma 2.7 of \cite{CoCr:DBun}, and so the result can be stated explicitly as follows:

\begin{PropSubSub}[{\cite[Lemma 2.7]{CoCr:DBun}}]\label{thm:pb_dbun}
Let $\underline{E} = (E,q,\sigma,\zeta,\lambda)$ be a differential bundle over $M$, and for each $n \in \NN$, let $q^{(n)}$ denote the $n$-th fibre power of $q$.  Suppose that $f:N \rightarrow M$ is a morphism in $\X$ such that the tangential pullback \pref{def:tangent_limit} of $q^{(n)}$ along $f$ exists in $\X$ for each $n$.  Then the pullback $f^*(q):N \times_M E \rightarrow N$ of $q$ along $f$ carries the structure of a differential bundle $f^*(\underline{E})$ over $N$.  Further, 
\begin{equation}\label{eq:cart_arr}(f',f):f^*(\underline{E}) \rightarrow \underline{E}\end{equation}
is a linear morphism of differential bundles, where $f':N \times_M E \rightarrow E$ is the pullback of $f$ along $q$.  Moreover, the morphism \eqref{eq:cart_arr} in $\DBun$ is a cartesian arrow, over $f$, with respect to functor $\mathsf{cod}:\DBun \rightarrow \X$ that sends each differential bundle $(F,r,\sigma',\zeta',\lambda')$ to the codomain of $r$.
\end{PropSubSub}

\begin{ParSubSub}\label{par:pb_dbun}
If the assumptions of \ref{thm:pb_dbun} are satisfied, then we say that the differential bundle $\underline{E}$ \textbf{has a pullback along $f$}, or that \textbf{the pullback of $\underline{E}$ along $f$ exists}.  Let us assume for the moment that this is the case.  We then call $f^*(\underline{E})$ the \textbf{pullback of $\underline{E}$ along $f$}.  More precisely, a \textit{pullback of $\underline{E}$ along $f$} is a differential bundle $\underline{F}$ equipped with a cartesian arrow $\alpha:\underline{F} \rightarrow \underline{E}$ over $f$ with respect to $\mathsf{cod}:\DBun \rightarrow \X$.  Any two such differential bundles are isomorphic, so we call each of them \textit{the pullback}.  Moreover, given any two pullbacks $(\underline{F},\alpha)$ and $(\underline{G},\beta)$ of $\underline{E}$ along $f:N \rightarrow M$, there is a unique isomorphism $\underline{F} \cong \underline{G}$ in $\DBun_N$ that commutes with the cartesian arrows $\alpha,\beta$ in $\DBun$.  Hence, in view of the construction of $f^*(\underline{E})$ in \ref{thm:pb_dbun} we find that for any given pullback $(\underline{F},\alpha)$ of $\underline{E}$ along $f$, if we write $\underline{F} = (F,r,\sigma',\zeta',\lambda')$ and $\alpha = (\alpha_1,\alpha_2)$, then the square
$$
\xymatrix{
F \ar[d]_{r} \ar[r]^{\alpha_1} & E \ar[d]^q\\
N \ar[r]_{\alpha_2\:=\:f}      & M
}
$$
is a pullback in $\X$, which we call \textbf{the pullback in $\X$ underlying $(\underline{F},\alpha)$}, or simply \textit{the pullback in $\X$ underlying $\alpha$}.  Still assuming that the pullback $f^*(\underline{E})$ exists, if $\alpha:\underline{F} \rightarrow \underline{E}$ is any given morphism over $f$ in $\DBun$, then $\alpha$ is cartesian if and only if the commutative square underlying $\alpha$ is a pullback square in $\X$; indeed, any such morphism $\alpha$ factors through the pullback $f^*(\underline{E}) \rightarrow \underline{E}$ by way of unique morphism $\phi:\underline{F} \rightarrow f^*(\underline{E})$ in $\DBun_N$, and if the square underlying $\alpha$ is a pullback it follows that $\phi$ is an isomorphism in $\X \slash N$ and hence in $\DBun_N$.  Note also that if two pullbacks $(\underline{F},\alpha)$ and $(\underline{G},\beta)$ of $\underline{E}$ along $f$ have the same underlying pullback in $\X$, then $\underline{F}$ and $\underline{G}$ are \textit{concretely isomorphic} as differential bundles over $N$ \pref{eq:ci}.
\end{ParSubSub}

\section{Review of connections in tangent categories}

The paper \cite{CoCr:Conn} defines a notion of connection on a differential bundle  $\underline{E} = (E,q,\sigma,\zeta,\lambda)$ over an object $M$, where the authors use a definition of the notion of differential bundle that is slightly different from the original definition \cite[Def. 2.3]{CoCr:DBun} in that it incorporates the following additional condition \cite[Def. 2.2]{CoCr:Conn}:

\begin{ParSub}[\textbf{Basic Condition}]\label{par:std_assn}
\textit{For all natural numbers $m,n \in \NN$, there is an associated tangential fibre product
$$\pover{E}{m}{M} \times_M\: T_nM$$
of $q^{(m)}:\pover{E}{m}{M} \rightarrow M$ and $p^{(n)}_M:T_nM \rightarrow M$ in $\X$ \pref{def:tangent_limit}.}
\end{ParSub}

\noindent Rather than building this condition into the definition of differential bundle, we will retain the original definition \cite[Def. 2.3]{CoCr:DBun}, which we reviewed in \ref{par:dbun}.

\begin{ParSub}\label{par:exa_std_assn}
It is noted in \cite[Example 2.3]{CoCr:Conn} that Basic Condition \ref{par:std_assn} is satisfied by wide classes of examples of differential bundles, including the following:
\begin{enumerate}
\item The tangent bundle $\underline{T}M$ always satisfies Basic Condition \ref{par:std_assn}.
\item If $\X = \Mf$ is the category of smooth manifolds and $\underline{E} = (E,q,\sigma,\zeta,\lambda)$ is a vector bundle, considered as a differential bundle in $\X$ \pref{exa:vb}, then $\underline{E}$ satisfies Basic Condition \ref{par:std_assn}.
\end{enumerate}
\end{ParSub}

Let us fix a differential bundle $\underline{E} = (E,q,\sigma,\zeta,\lambda)$, over $M$, satisfying Basic Condition \ref{par:std_assn}.  Then, in view of \ref{par:pb_dbun}, the following pullbacks of differential bundles exist: (1) the pullback $p_M^*(\underline{E})$ of $\underline{E}$ along $p_M:TM \rightarrow M$, and (2) the pullback $q^*(\underline{T}M)$ of $\underline{T}M$ along $q:E \rightarrow M$.  Therefore the object $E \times_M TM$ of $\X$ carries \textit{two} differential bundle structures, namely $p_M^*(\underline{E})$ and $q^*(\underline{T}M)$, over $TM$ and over $E$, respectively.

\begin{DefSub}[Cockett and Cruttwell \cite{CoCr:Conn}]\label{def:cx}
\emptybox
\begin{enumerate}
\item A \textbf{vertical connection} on $\underline{E}$ is a morphism $K:TE \rightarrow E$ in $\X$ such that $K$ is a retraction of $\lambda:E \rightarrow TE$ and satisfies the following axioms:
\begin{description}
\item[(C.1)] $(K,p_M):T\underline{E} \rightarrow \underline{E}$ is a linear morphism of differential bundles;
\item[(C.2)] $(K,q):\underline{T}E \rightarrow \underline{E}$ is a linear morphism of differential bundles.
\end{description}
\item A \textbf{horizontal connection} on $\underline{E}$ is a morphism\footnote{Whereas in \cite{CoCr:Conn} the domain of $H$ is written as $TM \times_M E$, we write it herein as $E \times_M TM$ for consistency with certain equivalent formulations developed in this paper.} $H:E \times_M TM \rightarrow TE$ in $\X$ such that $H$ is a section of $U = \langle p_E, T(q) \rangle$ and satisfies the following axioms:
\begin{description}
\item[(C.3)] $(H,1_E):q^*(\underline{T}M) \rightarrow \underline{T}E$ is a linear morphism of differential bundles;
\item[(C.4)] $(H,1_{TM}):p_M^*(\underline{E}) \rightarrow T\underline{E}$ is a linear morphism of differential bundles.
\end{description}
\item A \textbf{connection} on $\underline{E}$ is a pair $(K,H)$ consisting of a vertical connection $K$ and a horizontal connection $H$, on $\underline{E}$, such that the following axioms hold:
\begin{description}
\item[(C.5)] $HK = \pi_1 q \zeta\;:\;E \times_M TM \rightarrow E$;
\item[(C.6)] $\langle K,p_E\rangle\mu \;+\; UH \;=\; 1_{TE}$
\end{description}
where $\mu:E \times_M E \rightarrow TE$ is as defined in \ref{par:dbun} and we write $+$ in infix notation to denote the binary operation defined as follows:  Apply the finite-product-preserving functor $(\X \slash E)((TE,p_E),-):\X \slash E \rightarrow \Set$ to the commutative monoid $(TE,p_E,+_E,0_E)$ in order to obtain a commutative monoid structure on the hom-set $(\X \slash E)((TE,p_E),(TE,p_E))$.
\item A connection $(K,H)$ on $\underline{E}$ is said to be an \textbf{affine connection} on $M$ if $\underline{E} = \underline{T}M$.
\end{enumerate}
\end{DefSub}

\begin{ParSub}\label{par:uniq_hc_vc}
Given a vertical connection $K$ on $\underline{E}$, there is at most one horizontal connection $H$ such that $(K,H)$ is a connection on $\underline{E}$.  Indeed, this is proved in \cite[Prop. 5.10]{CoCr:Conn}, where it is also proved that for a given horizontal connection $H$ on $\underline{E}$ there is at most one vertical connection $K$ such that $(K,H)$ is a connection on $\underline{E}$.
\end{ParSub}

\begin{ThmSub}[{\cite[Prop. 5.12, Prop. 5.13]{CoCr:Conn}}]\label{thm:cocr_vhthm}
\emptybox
\begin{enumerate}
\item Let $K$ be a vertical connection on $\underline{E}$.  Suppose that $\X$ has negatives \pref{par:tcats}, and suppose that $\underline{E}$ carries at least one horizontal connection.  Then there is an associated horizontal connection $H$ on $\underline{E}$ such that $(K,H)$ is a connection.
\item  Let $H$ be a horizontal connection on $\underline{E}$, and suppose that $\X$ has negatives.  Then there is an associated vertical connection $K$ on $\underline{E}$ such that $(K,H)$ is a connection.
\end{enumerate}
\end{ThmSub}

\begin{ExaSub}[\textbf{The canonical affine connection on $\RR^n$}]\label{exa:can_aff_cx}
Continuing Example \ref{exa:sman}, let $\X = \Mf$ be the category of smooth manifolds, and let $V = \RR^n$, considered as an object of $\Mf$.  In view of \ref{thm:cocr_vhthm}, one can consult \cite[Example 3.6]{CoCr:Conn} for a treatment of the various affine connections on $V$; we now consider one particularly simple example.  There is a connection $(K,H)$ on $\underline{T}V$ in which\footnote{\label{fn:can_cx}Recall that we write the base-point $x$ in the leftmost coordinate of an element $(x,t,u,v)$ of $T^2V = V^4$ \pref{exa:sman}, thus writing the components of $(x,t,u,v)$ in an order opposite to that employed in \cite{CoCr:Conn}; cf. Footnote \ref{fn:l_to_r}.}
$$K = \langle \pi_1,\pi_4\rangle\;:\;T^2V = V^4 \longrightarrow V^2 = TV,\;\;\;\;(x,t,u,v) \mapsto (x,v).$$
The domain of the map $H$ is the fibre product $TV \times_V TV = V^2 \times_V V^2$ of two instances of $p_V = \pi_1:V^2 \rightarrow V$, so we may take $TV \times_V TV = V^3$ with associated projection $\pi_1:V^3 \rightarrow V$.  With this convention, $H$ is given as follows:
$$H\;:\;TV \times_V TV = V^3 \longrightarrow V^4 = T^2V,\;\;\;\;(x,t,u) \mapsto (x,t,u,0).$$
\end{ExaSub}

\section{Biproducts of differential bundles: Whitney sums}

In the sections that follow, we shall establish a useful equivalent formulation of connections in tangent categories.  As a key ingredient in this study of connections, we now examine closely the formation of products and biproducts of differential bundles over a fixed base.  The paper \cite{CoCr:DBun} constructs finite products of differential bundles over a fixed base object, under certain assumptions on the given tangent category $\X$ and the bundles involved.  In the present section, we instead treat the question of existence of the product of a given finite family of differential bundles over an object $M$ of an arbitrary tangent category.  But first let us begin with the following observation:

\begin{PropSub}\label{thm:dbunm_add}
The category $\DBun_M$ of differential bundles over $M$ is an additive category.  Hence finite products in $\DBun_M$ are biproducts \pref{par:add_cat}.
\end{PropSub}
\begin{proof}
By \ref{par:cat_dbuns}, there is a faithful `forgetful' functor $\DBun_M \rightarrow \A$, where $\A = \CMon(\X \slash M)$.  Hence, since $\A$ is an additive category \pref{par:cmonc} it suffices to show that for any pair of differential bundles $\underline{E} = (E,q,\sigma,\zeta,\lambda)$, $\underline{F} = (F,r,\sigma',\zeta',\lambda')$, the subset
$$\DBun_M(\underline{E},\underline{F}) \hookrightarrow \A(\underline{E},\underline{F})$$
is a submonoid, where we do not distinguish notationally between a differential bundle and its underlying commutative monoid.  The zero element of the commutative monoid $\A(\underline{E},\underline{F})$ is the composite $q\zeta': E \rightarrow F$, which is a morphism of differential bundles over $M$ since
$$\lambda T(q\zeta') = \lambda T(q)T(\zeta') = q0_MT(\zeta') = \zeta'\lambda'\;,$$
using \ref{par:dbun}(2).  Also, given linear morphisms of differential bundles $f,g:\underline{E} \rightarrow \underline{F}$, their sum $f + g$ in $\A(\underline{E},\underline{F})$ is the composite $E \xrightarrow{\langle f,g\rangle} F \times_M F \xrightarrow{\sigma'} F$, which we claim is a morphism of differential bundles over $M$.  Indeed, we have a diagram
\begin{equation}\label{eq:sum_mor_dbm}
\xymatrix{
E \ar[d]_\lambda \ar[r]^(.4){\langle f,g\rangle} & F \times_M F \ar[d]_{\lambda' \times \lambda'} \ar[r]^(.6){\sigma'} & F \ar[d]^{\lambda'}\\
TE \ar[r]_(.35){T(\langle f,g\rangle)} & T(F \times_M F) \ar[r]_(.65){T(\sigma')} & TF
}
\end{equation}
recalling that $T(F \times_M F)$ is a fibre product $TF \times_{TM} TF$ by \ref{par:dbun}(1), so that $T(\langle f,g\rangle)$ is the induced morphism $\langle T(f),T(g)\rangle:TE \rightarrow TF \times_{TM} TF$.  Hence the leftmost square commutes since $f$ and $g$ are morphisms of differential bundles over $M$, and the rightmost square commutes by \ref{par:dbun}(2).
\end{proof}

\begin{PropSub}\label{thm:t_on_db}
There is an additive functor
$$\DBun_M \rightarrow \DBun_{TM},\;\;\;\;\;\underline{E} \mapsto T\underline{E}$$
that sends each differential bundle $\underline{E}$ over $M$ to the differential bundle $T\underline{E}$ over $TM$ defined in \ref{exa:te} and sends each morphism $f$ to $T(f)$.
\end{PropSub}
\begin{proof}
Given $f:\underline{E} \rightarrow \underline{F}$ in $\DBun_M$, where $\underline{E} = (E,q,\sigma,\zeta,\lambda)$ and $\underline{F} = (F,r,\sigma',\zeta',\lambda')$, we deduce that $T(f):T\underline{E} \rightarrow T\underline{F}$ is a morphism in $\DBun_{TM}$ since $T(f)T(r) = T(fr) = T(q)$ and the diagram
$$
\xymatrix{
TE \ar[d]_{T(f)} \ar[r]^{T(\lambda)} & T^2 E \ar[d]_{T^2(f)} \ar[r]^{c_E} & T^2 E \ar[d]^{T^2(f)}\\
TF \ar[r]_{T(\lambda')} & T^2F \ar[r]_{c_F} & T^2 F
}
$$
commutes, using the naturality of $c$ and the fact that $f$ is a morphism of differential bundles over $M$.  Given a pair of morphisms $f,g:\underline{E} \rightarrow \underline{F}$ in $\DBun_M$, their sum $f + g$ is the top row in the diagram \eqref{eq:sum_mor_dbm}, so $T(f + g)$ is the bottom row in \eqref{eq:sum_mor_dbm}.  But we noted in the proof of the preceding Proposition that the morphism $T(\langle f,g \rangle)$ appearing in that same diagram \eqref{eq:sum_mor_dbm} can be described equally as $\langle T(f),T(g)\rangle:TE \rightarrow TF \times_{TM} TF$, so the bottom row in \eqref{eq:sum_mor_dbm} is precisely the sum $T(f) + T(g)$ in $\DBun_{TM}(T\underline{E},T\underline{F})$.  Also, the zero element of $\DBun_M(\underline{E},\underline{F})$ is the composite $q\zeta':E \rightarrow F$, which is sent by $T$ to the zero element $T(q)T(\zeta')$ of $\DBun_{TM}(T\underline{E},T\underline{F})$.
\end{proof}

\begin{LemSub}\label{thm:dbun_pr_fp}
The forgetful functor $\DBun_M \rightarrow \X \slash M$ preserves finite products.
\end{LemSub}
\begin{proof}
In view of the proof of \ref{thm:dbunm_add}, we know that the forgetful functor $\DBun_M \rightarrow \CMon(\X \slash M)$ is additive and hence preserves finite products (by \ref{par:add_cat}).  Hence we may invoke \ref{par:cmon_pr_fp} to deduce that the composite functor $\DBun_M \rightarrow \CMon(\X \slash M) \rightarrow \X \slash M$ preserves finite products.
\end{proof}

\begin{PropSub}\label{thm:prod_dbun}
Let $(\underline{E_i})_{i \in I}$ be a finite family of differential bundles over a fixed object $M$ in an arbitrary tangent category $\X$, and write $\underline{E_i} = (E_i,q_i,\sigma_i,\zeta_i,\lambda_i)$.  Then a product $\prod_{i \in I}\underline{E_i}$ in $\DBun_M$ is equivalently given by a tangential fibre product $(E,q) = \prod_{i \in I}^M(E_i,q_i)$ in $\X$ \pref{def:tangent_limit} and, for each $n \in \NN$, a specified $n$-th tangential fibre power of $q$ \pref{def:tangent_limit}.
\end{PropSub}
\begin{proof}
First suppose we are given a tangential fibre product $(E,q) = \prod_{i \in I}^M(E_i,q_i)$ in $\X$ that has finite tangential fibre powers in $\X$.  By \ref{thm:bipr_cmonc}, $(E,q)$ carries the structure of a product $(E,q,\sigma,\zeta) = \prod_{i \in I}(E_i,q_i,\sigma_i,\zeta_i)$ in $\CMon(\X \slash M)$.  By hypothesis, the product projections $\pi_i:(E,q) \rightarrow (E_i,q_i)$ in $\X \slash M$ $(i \in I)$ are sent by $T$ to a product cone $T(\pi_i):(TE,T(q)) \rightarrow (TE_i,T(q_i))$ in $\X \slash TM$.  But in view of \ref{par:dbun}(2) we can form composite morphisms
$$(E,q0_M) \xrightarrow{\pi_i} (E_i,q_i0_M) \xrightarrow{\lambda_i} (TE_i,T(q_i))\;\;\;\;(i \in I)$$
in $\X \slash TM$, which therefore induce a morphism $\lambda:(E,q0_M) \rightarrow (TE,T(q))$.  It is now a straightforward exercise to show that the underlying morphism $\lambda:E \rightarrow TE$ makes $(E,q,\sigma,\zeta,\lambda)$ a differential bundle over $M$, and indeed a product $\prod_{i \in I}\underline{E_i}$ in $\DBun_M$; we shall not dwell on these calculations here, since this was treated in \cite[\S 5]{CoCr:DBun} under slightly different assumptions.

The converse, however, is completely new, and we now provide a full proof.  Suppose we are given a product $\prod_{i \in I}\underline{E_i}$ in $\DBun_M$.  Let us abuse notation by writing $T:\DBun_M \rightarrow \DBun_{TM}$ to denote the additive functor defined in \ref{thm:t_on_db}.  For each $n \in \NN$, we can repeatedly invoke \ref{thm:t_on_db} and by composition define an additive functor $T^n:\DBun_M \rightarrow \DBun_{T^nM}$, given on objects by $\underline{E} \mapsto T^n\underline{E}$.  Moreover we obtain a commutative diagram
\begin{equation}\label{eq:lift_tn}
\xymatrix{
\DBun_M \ar[d] \ar[r]^{T^n} & \DBun_{T^nM} \ar[d]\\
\X \slash M \ar[r]_{T^n}    & \X / T^nM
}
\end{equation}
in which the left and right sides are the forgetful functors and the bottom side is the evident functor induced by $T^n:\X \rightarrow \X$.  But since the top side is an additive functor and hence preserves finite products \pref{par:add_cat}, we can invoke the preceding lemma \pref{thm:dbun_pr_fp} to deduce that the clockwise composite in \eqref{eq:lift_tn} preserves finite products.  By the same lemma, the forgetful functor $\DBun_M \rightarrow \X \slash M$ preserves finite products and hence sends the finite product $\prod_{i \in I}\underline{E_i}$ to a fibre product $(E,q) = \prod_{i \in I}^M(E_i,q_i)$ in $\X$ that is preserved by $T^n$ (since the composite in \eqref{eq:lift_tn} preserves finite products).  But since $(E,q)$ underlies a differential bundle, namely $\prod_{i \in I}\underline{E_i}$, we know that $q$ has finite tangential fibre powers, by \ref{par:dbun}(1).

Using \ref{thm:bipr_cmonc}, it is straightforward to show that the above assignments are mutually inverse.
\end{proof}

\begin{CorSub}
Every differential bundle over $M$ has finite powers in $\DBun_M$.
\end{CorSub}

\begin{ParSub}[\textbf{Notation and terminology}]\label{par:notn_pr_dbun}
Given a finite family of differential bundles $(\underline{E_i})_{i \in I}$ over $M$ with a product in $\DBun_M$, we will denote this product by $\prod_{i \in I}^M\underline{E_i}$.  Such a product carries the structure of a biproduct in $\DBun_M$ \pref{thm:dbunm_add}, which we shall denote by $\bigoplus_{i \in I}^M\underline{E_i}$.  We call each biproduct in $\DBun_M$ a \textbf{biproduct of differential bundles over $M$} or a \textbf{Whitney sum}.  We also employ the notations $\underline{E_1} \times_M \underline{E_2} \times_M ... \times_M \underline{E_n}$ and $\underline{E_1} \oplus_M \underline{E_2} \oplus_M ... \oplus_M \underline{E_n}$ in the case that $I = \{1,...,n\}$.

Given a differential bundle $\underline{E} = (E,q,\sigma,\zeta,\lambda)$ over $M$ and a natural number $n$, we denote the $n$-th power of $\underline{E}$ in $\DBun_M$ by $\pover{\underline{E}}{n}{M}$.  The object of $\X \slash M$ underlying $\pover{\underline{E}}{n}{M}$ is the $n$-th fibre power of $q$ in $\X$, which we denote by $q^{(n)}:\pover{E}{n}{M} \rightarrow M$ \pref{par:fi_pr}.
\end{ParSub}

We shall later make use of the following corollary to \ref{thm:prod_dbun}:

\begin{CorSub}\label{thm:u_bipr_dbun}
Let $(\underline{E_i})_{i \in I}$ be a finite, nonempty family of differential bundles over $M$, written as $\underline{E_i} = (E_i,q_i,\sigma_i,\zeta_i,\lambda_i)$, and suppose that a product $\prod_{i \in I}^M \underline{E_i}$ exists in $\DBun_M$.  Let $E$ be an object of $\X$, and for each $i \in I$, let $\pi_i:E \rightarrow E_i$ be a morphism in $\X$.  Then the following are equivalent:
\begin{enumerate}
\item $E$ underlies a product of differential bundles $\prod_{i \in I}^M \underline{E_i}$ over $M$ with projections $\pi_i$ $(i \in I)$.
\item There exists a (necessarily unique) morphism $q:E \rightarrow M$ in $\X$ such that the morphisms $\pi_i$ $(i \in I)$ present $(E,q)$ as a fibre product of $(q_i:E_i \rightarrow M)_{i \in I}$.
\end{enumerate}
Further, a product $\prod_{i \in I}^M \underline{E_i}$ with underlying object $E$ and projections $\pi_i$ is unique up to \textit{concrete isomorphism} \pref{eq:ci} if it exists.
\end{CorSub}

\begin{DefSub}\label{def:set_has_fp}
Given a set of objects $\E \subseteq \ob\C$ in a category $\C$, we say that $\E$ \textbf{has finite products in $\C$} if for any finite set $I$ and any $I$-indexed family $(E_i)_{i \in I}$ of objects in $\E$, not necessarily distinct, there is an associated product $\prod_{i \in I}E_i$ in $\C$.  If $\E$ is a set of differential bundles over $M$, then we say that \textbf{$\E$ has finite products over $M$} if $\E$ has finite products in $\DBun_M$.
\end{DefSub}

\begin{PropSub}\label{thm:set_has_fp}
For each $i \in I = \{1,...,n\}$, let $\underline{E_i} = (E_i,q_i,\sigma_i,\zeta_i,\lambda_i)$ be a differential bundle over $M$.  Then the following are equivalent:
\begin{enumerate}
\item The set of differential bundles $\{\underline{E_i} 
\mid i \in I\}$ has finite products over $M$.
\item For all $m_1, m_2, ..., m_n \in \NN$, there is a tangential fibre product $\prod_{i \in I}^M\pover{E_i}{m_i}{M}$ of $(q_i^{(m_i)})_{i \in I}$ in $\X$ \pref{def:tangent_limit}, where $q_i^{(m_i)}:\pover{E_i}{m_i}{M} \rightarrow M$ denotes the $m_i$-th fibre power of $q_i$ \pref{par:dbun}.
\end{enumerate}
\end{PropSub}
\begin{proof}
We know that for each $i \in I$, the finite fibre powers of $q_i$ are tangential \pref{par:dbun}, so this follows from \ref{thm:prod_dbun}.
\end{proof}

\begin{CorSub}\label{thm:bc_iff_etm_has_fp}
A differential bundle $\underline{E}$ over $M$ satisfies Basic Condition \ref{par:std_assn} if and only if the set of differential bundles $\{\underline{E},\;\underline{T}M\}$ has finite products over $M$.
\end{CorSub}

\section{The partial bundles of a biproduct}\label{sec:pbun}

\begin{ParSub}\label{par:given_fam}
In the present section, we let $\underline{E} = \prod_{i \in I}^M\underline{E_i}$ denote a given finite product of differential bundles $\underline{E_i} = (E_i,q_i,\sigma_i,\zeta_i,\lambda_i)$ over $M$ in $\X$, and we suppose that the set of differential bundles $\{\underline{E_i} \mid i \in I\}$ has finite products over $M$ \pref{def:set_has_fp}.  Let us write $\underline{E} = (E,q,\sigma,\zeta,\lambda)$.
\end{ParSub}

We now show that $E$ carries the structure of a differential bundle over $E_j$, for each $j \in I$, which we call the \textit{$j$-th partial bundle of the product $\prod_{i \in I}^M\underline{E_i}$}.  As we shall see in \ref{thm:par_bun}(2) and \ref{par:add_jpb}, the fibrewise addition operation in the $j$-th partial bundle is given by leaving the $j$-th coordinate fixed and adding the respective $i$-th coordinates for each $i \neq j$.

To this end, if we let $j \in I$ then we can form a product of differential bundles $\prod_{i \in I\backslash\{j\}}^M\underline{E_i}$
whose underlying object of $\X \slash M$ is a fibre product $\prod_{i \in I\backslash\{j\}}^ME_i$ over $M$ \pref{thm:dbun_pr_fp}.  We obtain a pullback square
\begin{equation}\label{eq:pb_sq}
\xymatrix{
*!<3.5ex,0ex>{E = \prod_{i \in I}^M E_i} \ar[d]_{\pi_j} \ar[r]^{\pi_{I\backslash\{j\}}} &*!<-2ex,0ex>{\prod_{i \in I\backslash\{j\}}^ME_i} \ar[d]^{\text{proj}}\\
E_j \ar[r]_{q_j}                                                  & M
}
\end{equation}
in which $\pi_{I\backslash\{j\}}$ and $\text{proj}$ denote the evident projection morphisms.  We now show that this pullback square \textit{underlies} a pullback of differential bundles, in the sense of \ref{par:pb_dbun}:

\begin{PropSub}\label{thm:parbun}
For each $j \in I$, the differential bundle $\prod_{i \in I\backslash\{j\}}^M \underline{E_i}$ has a pullback
\begin{equation}\label{eq:jpbun}q^*_j\left(\prod\nolimits_{i \in I\backslash \{j\}}^M\underline{E_i}\right)\end{equation}
along $q_j:E_j \rightarrow M$ whose underlying pullback in $\X$ is the square \eqref{eq:pb_sq}.
\end{PropSub}
\begin{proof}
Let $\underline{F} = \prod_{i \in I'}^{\scriptscriptstyle M}\underline{E_i}$ where $I' = I\backslash\{j\}$.  Hence the object $(F,r)$ of $\X \slash M$ underlying $\underline{F}$ is a fibre product $F = \prod_{i \in I'}^{\scriptscriptstyle M} E_i \rightarrow M$ in $\X$ \pref{thm:dbun_pr_fp}.  For each $n \in \NN$, we deduce by \ref{thm:set_has_fp} that there are tangential fibre products 
$$G = \prod\nolimits_{i \in I'}^M \pover{E_i}{n}{M}\;\;\;\;\;\text{and}\;\;\;\;\;H = \prod\nolimits_{i \in I}^M\pover{E_i}{m_i}{M}\;,$$
where we define $m_j = 1$ and $m_i = n$ for each $i \in I'$.  We can write $H$ as a fibre product $H = E_j \times_M G$, and we now deduce that the latter (binary) fibre product is tangential \pref{def:tangent_limit}.  But the projection $G \rightarrow M$ is an $n$-th fibre power $r^{(n)}:\pover{F}{n}{M} \rightarrow M$ of $(F,r)$.  Hence the projection $\pi_1:E_j \times_M G \rightarrow E_j$ is a tangential pullback of $r^{(n)}$ along $q_j:E_j \rightarrow M$ \pref{def:tangent_limit}.  Therefore the needed pullback bundle $q_j^*(\underline{F})$ exists \pref{par:pb_dbun}, and we may take its underlying pullback in $\X$ to be the square \eqref{eq:pb_sq}.
\end{proof}

\begin{DefSub}\label{def:jth_pb}
Given data as in \ref{par:given_fam}, let $j \in I$.  By \ref{thm:parbun}, the projection $\pi_j:E \rightarrow E_j$ underlies a differential bundle $q_j^*(\prod_{i \in I\backslash\{j\}}^M\underline{E_i})$, which we call the \textbf{$j$-th partial bundle} of the (bi)product $\underline{E} = \prod_{i \in I}^M\underline{E_i} = \oplus_{i \in I}^M\underline{E_i}$ \pref{par:notn_pr_dbun}.  More precisely, given an arbitrary differential bundle $\underline{P}$ over $E_j$, we say that $\underline{P}$ is a \textit{$j$-th partial bundle} (of the given product) if there exists a cartesian arrow $\alpha:\underline{P} \rightarrow \prod_{i \in I\backslash\{j\}}^M\underline{E_i}$ in $\DBun$ whose underlying pullback in $\X$ is the square \eqref{eq:pb_sq}, cf. \ref{par:pb_dbun}.  By \ref{thm:parbun}, a $j$-th partial bundle exists, and by \ref{par:pb_dbun} any two $j$-th partial bundles are \textit{concretely isomorphic} in the terminology of \ref{eq:ci}.  Hence we call any $j$-th partial bundle of $\prod_{i \in I}^M\underline{E_i}$ \textit{the $j$-th partial bundle}, with the understanding that this notion is defined only up to concrete isomorphism.
\end{DefSub}

\begin{RemSub}\label{rem:jth_pb}
It is immediate from Definition \ref{def:jth_pb} that if $\underline{P} = (P,q^j,\sigma^j,\zeta^j,\lambda^j)$ is a $j$-th partial bundle  of the product $\underline{E} = \prod_{i \in I}^M\underline{E_i}$, then $P = E$, and the associated morphism $q^j:P \rightarrow E_j$ is the projection $\pi_j:E \rightarrow E_j$.  Further, the cartesian arrow $\alpha = (\alpha_1,\alpha_2)$ in \ref{def:jth_pb} necessarily consists of the top and bottom sides of the pullback square \eqref{eq:pb_sq}, i.e. $\alpha_1 = \pi_{I\backslash\{j\}}$ and $\alpha_2 = q_j$.  Since $\alpha$ is thus uniquely determined, we obtain the following characterization of the $j$-th partial bundle:
\end{RemSub}

\begin{PropSub}\label{thm:par_bun_lem}
Suppose we are given data as in \ref{par:given_fam}, and let $j \in I$.  Then a differential bundle $\underline{P}$ over $E_j$ is a $j$-th partial bundle of $\underline{E} = \prod_{i \in I}^M\underline{E_i}$ if and only if the following conditions hold: (1) The object of $\X \slash E_j$ underlying $\underline{P}$ is the projection $\pi_j:E \rightarrow E_j$, and (2) the commutative square \eqref{eq:pb_sq} underlies a linear morphism of differential bundles $(\pi_{I\backslash\{j\}},q_j):\underline{P} \rightarrow \prod_{i \in I\backslash\{j\}}^M\underline{E_i}$.
\end{PropSub}
\begin{proof}
By \ref{rem:jth_pb}, if $\underline{P}$ is a $j$-th partial bundle then the given conditions hold.  Conversely, if the given conditions hold, then since the commutative square underlying $\alpha = (\pi_{I\backslash\{j\}},q_j):\underline{P} \rightarrow \prod_{i \in I\backslash\{j\}}^M\underline{E_i}$ is a pullback in $\X$, it follows that $\alpha$ is a cartesian arrow in $\DBun$, by \ref{par:pb_dbun} (and \ref{thm:parbun}), so $\underline{P}$ is a $j$-th partial bundle of the given product.
\end{proof}

This leads us to the following concrete characterization of the $j$-th partial bundle:

\begin{PropSub}\label{thm:par_bun}
Suppose that we are given data as in \ref{par:given_fam}, and let $j \in I$.
\begin{enumerate}
\item The $j$-th partial bundle of the product $\underline{E} = \prod_{i \in I}^M\underline{E_i}$ is a differential bundle $\underline{P}$ over $E_j$ that is uniquely characterized, up to concrete isomorphism \pref{def:jth_pb}, by the following statements:
\begin{enumerate}
\item the object of $\X \slash E_j$ underlying $\underline{P}$ is the projection $\pi_j:E \rightarrow E_i$, and
\item $(\pi_i,q_j):\underline{P} \rightarrow \underline{E_i}$ is a linear morphism of differential bundles, for each $i \in I\backslash\{j\}$.
\end{enumerate}
\item Letting $\underline{P} = (E,\pi_j,\sigma^j,\zeta^j,\lambda^j)$ be the $j$-th partial bundle of $\underline{E} = \prod_{i \in I}^M\underline{E_i}$, the following diagrams commute, where $\textnormal{proj}$ denotes the projection morphism.  For a given choice of fibre product $E \times_{E_j} E$, the commutativity of these diagrams uniquely characterizes $\sigma^j$, $\zeta^j$, and $\lambda^j$.
\begin{equation}\label{eq:pbun_d1}
\xymatrix{
E \ar[d]_{\pi_i} \ar[r]^{\lambda^j}                            & TE \ar[d]^{T(\pi_i)} & E \times_{E_j} E \ar[d]_{\pi_i \times \pi_i} \ar[r]^(.6){\sigma^j} & E \ar[d]^{\pi_i} & E_j \ar[d]_{q_j} \ar[r]^{\zeta^j} & E \ar[d]^{\pi_i} & \ar@{}[d]|{\displaystyle\;\;\;(i \in I\backslash\{j\})}\\
E_i \ar[r]_{\lambda_i}                                         & TE_i                 & E_i \times_M E_i \ar[r]_(.6){\sigma_i} & E_i & M \ar[r]_{\zeta_i}                & E_i & 
}
\end{equation}
\begin{equation}\label{eq:pbun_d2}
\xymatrix{
E \ar[d]_{\pi_j} \ar[r]^{\lambda^j} & TE \ar[d]^{T(\pi_j)} & E \times_{E_j} E \ar[dr]_{\textnormal{proj}} \ar[r]^(.6){\sigma^j} & E \ar[d]^{\pi_j} & E_j \ar[dr]_{1} \ar[r]^{\zeta^j} & E \ar[d]^{\pi_j} & \ar@{}[d]|{\phantom{\displaystyle\;\;\;(i \in I\backslash{j})}}\\
E_j \ar[r]_{0_{E_j}}                                     & TE_j                 &  & E_j &                                  & E_j &
}
\end{equation}
\item The zero section $\zeta^j:E_j \rightarrow E$ carried by the $j$-th partial bundle $\underline{P}$ is precisely the injection $\iota_j:\underline{E_j} \rightarrow \underline{E}$ associated to the biproduct $\underline{E} = \oplus_{i \in I}^M\underline{E_i}$.
\end{enumerate}
\end{PropSub}
\begin{proof}
Let us first show that any $j$-th partial bundle $\underline{P}$ satisfies conditions 1(a,b).  By \ref{thm:par_bun_lem} we know that (a) holds. For each $i \in I\backslash\{j\}$, we can apply \ref{thm:par_bun_lem} in order to obtain a composite morphism 
$$\underline{P} \xrightarrow{(\pi_{I\backslash\{j\}},q_j)} \prod\nolimits_{i \in I\backslash \{j\}}^M\underline{E_i} \xrightarrow{(\pi_i,1_M)} \underline{E_i}$$
in $\DBun$, but this composite is precisely $(\pi_i,q_j):\underline{P} \rightarrow \underline{E_i}$, so (b) holds.

Next note that if an arbitrary differential bundle $\underline{P} = (P,q^j,\sigma^j,\zeta^j,\lambda^j)$ over $E_j$ satisfies conditions 1(a,b), then the diagrams \eqref{eq:pbun_d1} commute since linear morphisms are additive \pref{par:cat_dbuns}, and the diagrams \eqref{eq:pbun_d2} commute since $\underline{P}$ is a differential bundle.  Since the fibre product $E = \prod_{i \in I}^M E_i$ is tangential \pref{thm:prod_dbun} and hence is preserved by $T$, the equations that assert commutativity of \eqref{eq:pbun_d1} and \eqref{eq:pbun_d2} uniquely characterize $\lambda^j$, $\sigma^j$, and $\zeta^j$ once the second fibre power $E \times_{E_j} E$ is chosen.

It follows that any differential bundle $\underline{P}$ satisfying conditions 1(a,b) is uniquely characterized up to concrete isomorphism by 1(a) together with the commutativity of the diagrams \eqref{eq:pbun_d1} and \eqref{eq:pbun_d2}.  Hence any two differential bundles over $E_j$ that satisfy 1(a,b) are concretely isomorphic.

Thus 1 and 2 are proved.  For 3, recall from \ref{par:add_cat} that the morphism $\iota_j:\underline{E_j} \rightarrow \underline{E} = \oplus_{i \in I}^M\underline{E_i}$ in $\DBun_M$ is uniquely defined by the equations $\iota_j\pi_j = 1_{E_j}$ and $\iota_j\pi_i = 0$ $(i \in I \backslash \{j\})$, where we write $0$ for the zero element of the commutative monoid $\DBun_M(\underline{E_j},\underline{E_i})$. Explicitly, this zero element is the composite $q_j\zeta_i$ that appears in \eqref{eq:pbun_d1}.  Hence if we substitute the morphism $\iota_j:E_j \rightarrow E$ in place of $\zeta^j$ within the rightmost diagrams in \eqref{eq:pbun_d1} and \eqref{eq:pbun_d2}, respectively, then the resulting two diagrams commute.  But by 2, the commutativity of these diagrams uniquely characterizes $\zeta^j$, so $\iota_j = \zeta^j$.
\end{proof}

\begin{ParSub}[\textbf{Addition in the $j$-th partial bundle}]\label{par:add_jpb}
By forming the $j$-th partial bundle of a product $\underline{E} = \prod_{i \in I}^M \underline{E_i}$, we equip $E$ with the structure of a commutative monoid $(E,\pi_j,\sigma^j,\zeta^j)$ in $\X \slash E_j$ \pref{thm:par_bun}.  Hence for each object $(X,x)$ of $\X \slash E_j$, the hom-set
$$(\X \slash E_j)((X,x),(E,\pi_j))$$
acquires the structure of a commutative monoid, whose addition operation we shall write as $+^j$.  We call $+^j$ the operation of \textbf{addition with respect to the $j$-th partial bundle} of the given product.

In order to characterize this addition operation $+^j$, first recall that for any object $X$ of $\X$, a morphism $f:X \rightarrow E = \prod^M_{i \in I}E_i$ is equivalently given by a family of morphisms $f_i = f\pi_i:X \rightarrow E_i$ such that $f_iq_i = f_kq_k:X \rightarrow M$ for all $i,k \in I$ \pref{par:fi_pr}.  Given morphisms $f,g:X \rightarrow E$ induced by $f_i,g_i:X \rightarrow E_i$ $(i \in I)$, respectively, if we assume that $f_j = g_j$ then upon setting $x = f_j = g_j$ we find that $f,g:(X,x) \rightarrow (E,\pi_j)$ in $\X \slash E_j$, so we can consider the sum
$$f +^j g\;:\;X \rightarrow E = \prod\nolimits^M_{i \in I}E_i\;.$$
In order to characterize this sum, first note that for each $i \in I$ we can construe $f_i$ and $g_i$ as morphisms $f_i,g_i:(X,x_M) \rightarrow (E_i,q_i)$ in $\X \slash M$, where $x_M = xq_j:X \rightarrow M$.  With this notation, we obtain the following:
\end{ParSub}

\begin{PropSub}\label{thm:addn_pb}
Given morphisms $f,g:X \rightarrow E = \prod_{i \in I}^ME_i$ in $\X$ such that $f_j = g_j$, the sum $f +^j g:X \rightarrow E$ is uniquely characterized by the following equations
$$
(f +^j g)\pi_i = \begin{cases}
                    \;f_i + g_i \::\:X \rightarrow E_i & (i \neq j)\\
                    f_j = g_j\::\:X \rightarrow E_j         & (i = j)                  
                 \end{cases}
$$
$(i \in I)$, where we write $+$ to denote the addition operation that is carried by the hom-set $(\X \slash M)((X,x_M),(E_i,q_i))$ and is induced by the commutative monoid $(E,q_i,\sigma_i,\zeta_i)$ in $\X \slash M$ underlying $\underline{E_i}$.
\end{PropSub}
\begin{proof}
It is straightforward to employ \ref{thm:par_bun}(2) to show that the given equations hold.  The needed uniqueness follows from the fact that the projections $\pi_i:E \rightarrow E_i$ $(i \in I)$ are jointly monic in $\X$, since $I \neq \emptyset$ \pref{par:fi_pr}.  
\end{proof}

\section{The biproduct decomposition determined by a connection}\label{sec:bipr_cx}

Let $\underline{E} = (E,q,\sigma,\zeta,\lambda)$ be a differential bundle over $M$, and assume that $\underline{E}$ satisfies Basic Condition \ref{par:std_assn}.  Let us begin by recalling the following proposition of Cockett and Cruttwell:

\newpage

\begin{PropSub}[{\cite[Prop. 5.8]{CoCr:Conn}}]\label{thm:cx_fp}
If $(K,H)$ is a connection on $\underline{E}$, then the following is a fibre product diagram in $\X$:
\begin{equation}\label{eq:cx_ld}
\xymatrix{
            & TE \ar[dl]_{p_E} \ar[d]|{T(q)} \ar[dr]^K & \\
E \ar[dr]_q & TM \ar[d]|{p_M}                          & E \ar[dl]^q\\
   & M                                                 &
}
\end{equation}
\end{PropSub}

We shall now show that this diagram underlies a biproduct of differential bundles $\underline{E} \oplus_M \underline{T}M \oplus_M \underline{E}$ whose first and second partial bundles are $\underline{T}E$ and $T\underline{E}$, respectively.  In section \ref{sec:cx_eq_form} we show that this leads to a useful equivalent formulation of connections in terms of just $K$, rather than both $K$ and $H$.  

\begin{LemSub}\label{thm:vcx_bipr1}
Let $K:TE \rightarrow E$ be a morphism in $\X$.  Then the following are equivalent:
\begin{enumerate}
\item The diagram \eqref{eq:cx_ld} is a fibre product diagram.
\item $TE$ underlies a biproduct of differential bundles
\begin{equation}\label{eq:lem_bipr}\underline{E} \oplus_M \underline{T}M \oplus_M \underline{E}\end{equation}
over $M$ with the following projections:
\begin{equation}\label{eq:3projs}\pi_1 = p_E:TE \rightarrow E,\;\;\pi_2 = T(q):TE \rightarrow TM,\;\;\pi_3 = K:TE \rightarrow E\end{equation}
\end{enumerate}
\end{LemSub}
\begin{proof}
By \ref{thm:bc_iff_etm_has_fp}, the set of differential bundles $\{\underline{E},\;\underline{T}M\}$ has finite products over $M$.  Therefore a product $\underline{E} \times_M \underline{T}M \times_M \underline{E}$ exists in $\DBun_M$, so this follows from \ref{thm:u_bipr_dbun} and \ref{thm:dbunm_add}.
\end{proof}

\begin{RemSub}\label{rem:kbp_un_ci}
Condition 2 in Lemma \ref{thm:vcx_bipr1} asserts the existence of a biproduct with certain further properties.  However, by \ref{thm:u_bipr_dbun}, a biproduct \eqref{eq:lem_bipr} with projections as in \eqref{eq:3projs} is unique up to \textit{concrete isomorphism} \pref{eq:ci}, if it exists.  Explicitly, if both $\underline{F}$ and $\underline{G}$ are biproducts of the form \eqref{eq:lem_bipr} and satisfy \eqref{eq:3projs}, then the objects $(F,r)$ and $(G,s)$ of $\X \slash M$ underlying $\underline{F}$ and $\underline{G}$ are identical, and the identity morphism on $F = G$ is an isomorphism $\underline{F} \cong \underline{G}$ in $\DBun_M$; further, the biproduct projections and injections associated to these biproducts $\underline{F},\underline{G}$ are identical.
\end{RemSub}

\begin{LemSub}\label{thm:vcx_bipr2}
Let $K:TE \rightarrow E$ be a retraction of the lift morphism $\lambda$ in $\X$, and suppose that $TE$ underlies a given biproduct of differential bundles $\underline{E} \oplus_M \underline{T}M \oplus_M \underline{E}$ with projections as given in \eqref{eq:3projs}.
\begin{enumerate}
\item $K$ is a vertical connection on $\underline{E}$ if and only if the first and second partial bundles of the given biproduct $\underline{E} \oplus_M \underline{T}M \oplus_M \underline{E}$ are $\underline{T}E$ and $T\underline{E}$, respectively.
\item If the equivalent conditions in 1 hold, then the injections of the given biproduct $\underline{E} \oplus_M \underline{T}M \oplus_M \underline{E}$ are as follows:
$$\iota_1 = 0_E:E \rightarrow TE,\;\;\iota_2 = T(\zeta):TM \rightarrow TE,\;\;\iota_3 = \lambda:E \rightarrow TE\;.$$
\end{enumerate}
\end{LemSub}
\begin{proof}
In order to prove 1, first recall that $\underline{T}E = (TE,p_E,+_E,0_E,\ell_E)$ is a differential bundle over $E$ whose underlying object of $\X \slash E$ is $(TE,p_E)$.  Hence we deduce by \ref{thm:par_bun} that $\underline{T}E$ is the first partial bundle of the given biproduct \eqref{eq:lem_bipr} if and only if the following conditions hold:
\begin{enumerate}
\item[(a)] $(T(q),q):\underline{T}E \rightarrow \underline{T}M$ is a linear morphism of differential bundles;
\item[(b)] $(K,q):\underline{T}E \rightarrow \underline{E}$ is a linear morphism of differential bundles.
\end{enumerate}
Similarly, we deduce by \ref{thm:par_bun} that $T\underline{E} = (TE,T(q),T(\sigma),T(\zeta),T(\lambda))$ is the second partial bundle of the biproduct \eqref{eq:lem_bipr} if and only if the following conditions hold:
\begin{enumerate}
\item[(c)] $(p_E,p_M):T\underline{E} \rightarrow \underline{E}$ is a linear morphism of differential bundles;
\item[(d)] $(K,p_M):T\underline{E} \rightarrow \underline{E}$ is a linear morphism of differential bundles.
\end{enumerate}
By definition, $K$ is a vertical connection if and only if (b) and (d) hold \pref{def:cx}.  Hence in order to prove 1 it suffices to show that (a) and (c) always hold.  Indeed, (a) holds as a consequence of the the naturality of $p$ and of $\ell$.  Also, (c) holds, by the naturality of $p$ and the fact that
$$
\xymatrix{
TE \ar[d]_{p_E} \ar[r]^{T(\lambda)} & T^2E \ar[r]^{c_E} \ar[d]_{p_{TE}} & T^2E \ar[d]^{T(p_E)}\\
E \ar[r]_{\lambda}                  & TE \ar@{=}[r]                     & TE 
}
$$ 
commutes, by \ref{par:tcats}(1,2) and the naturality of $p$.

Suppose that the equivalent conditions in 1 hold.  We deduce by \ref{thm:par_bun} that the injections $\iota_1$ and $\iota_2$ of the biproduct \eqref{eq:lem_bipr} are the zero sections of the first and second partial bundles $\underline{T}E$ and $T\underline{E}$, respectively, i.e. $\iota_1 = 0_E$ and $\iota_2 = T(\zeta)$.  The morphism $\iota_3:\underline{E} \rightarrow \underline{E} \oplus_M \underline{T}M \oplus_M \underline{E}$ is uniquely characterized by the equations $\iota_3\pi_1 = 0$, $\iota_3\pi_2 = 0$, and $\iota_3\pi_3 = 1_{\underline{E}}$, where we write $0$ to denote the zero morphisms $0:\underline{E} \rightarrow \underline{E}$ and $0:\underline{E} \rightarrow \underline{T}M$ in the additive category $\DBun_M$.  Hence, using explicit formulae for these zero morphisms, we find that the underlying morphism $\iota_3:E \rightarrow TE$ in $\X$ satisfies the equations $\iota_3p_E = q\zeta:E \rightarrow E$, $\iota_3T(q) = q0_M:E \rightarrow TM$, and $\iota_3K = 1_E:E \rightarrow E$.  But by \ref{par:dbun}(2,3) and the fact that $K$ is a retraction of $\lambda$, we know that $\lambda:E \rightarrow TE$ satisfies the analogous equations $\lambda p_E = q\zeta$, $\lambda T(q) = q0_M$, and $\lambda K = 1_E$.  Hence, since the projections $\pi_1 = p_E$, $\pi_2 = T(q)$, $\pi_3 = K$ are jointly monic in $\X$ (by \ref{par:fi_pr}), we deduce that $\iota_3 = \lambda$. 
\end{proof}

In view of \ref{thm:cx_fp}, the following is an immediate corollary to Lemmas \ref{thm:vcx_bipr1} and \ref{thm:vcx_bipr2}.

\begin{ThmSub}\label{thm:bipr_cx}
Let $(K,H)$ be a connection on $\underline{E}$.
\begin{enumerate}
\item $TE$ underlies a biproduct of differential bundles $\underline{E} \oplus_M \underline{T}M \oplus_M \underline{E}$ over $M$ whose projections and injections $\pi_i$, $\iota_i$ $(i = 1,2,3)$ are as follows (in ascending order\footnote{\label{fn:cx_ord}Note that here we have chosen a specific order in which to enumerate the three projections $p_E,T(q),K$.  The reasons for our choice are elaborated in Footnote \ref{fn:pr_ord_cx} with reference to Example \ref{exa:tb_ccx}.} by the index $i$, from left to right):
$$
\xymatrix{
  & & TE \ar@<-.5ex>[dll]_(.6){p_E} \ar@<-.5ex>[d]_(.6){T(q)} \ar@<.5ex>[drr]^(.6){K} & & \\
E \ar@<-.5ex>[urr]_(.4){0_E} & & TM \ar@<-.5ex>[u]_(.4){T(\zeta)} & & E. \ar@<.5ex>[ull]^(.4){\lambda}
}
$$
\item The first and second partial bundles of the biproduct $\underline{E} \oplus_M \underline{T}M \oplus_M \underline{E}$ in 1 are
$$\underline{T}E = (TE,p_E,+_E,0_E,\ell_E)\;\;\;\;\;\;\text{and}\;\;\;\;\;\;T\underline{E} = (TE,T(q),T(\sigma),T(\zeta),T(\lambda)),$$ 
respectively.
\end{enumerate}
\end{ThmSub}

Note that the conclusions of this theorem make no reference to the horizontal connection $H$ and, by Lemmas \ref{thm:vcx_bipr1} and \ref{thm:vcx_bipr2}, they hold for any vertical connection $K$ having a certain additional property.  We now make this explicit.

\begin{DefSub}
An \textbf{effective vertical connection} on $\underline{E}$ is a vertical connection $K:TE \rightarrow E$ on $\underline{E}$ such that \eqref{eq:cx_ld} is a fibre product diagram in $\X$.
\end{DefSub}

\begin{ThmSub}\label{thm:evcx_bipr}
If $K:TE \rightarrow E$ is an effective vertical connection on $\underline{E}$, then statements 1 and 2 of Theorem \ref{thm:bipr_cx} hold.
\end{ThmSub}

\begin{DefSub}
We call the differential bundle 
$$\underline{E} \oplus_M \underline{T}M \oplus_M \underline{E} = (TE,\widehat{q},\widehat{\sigma},\widehat{\zeta},\widehat{\lambda})$$
in Theorem \ref{thm:bipr_cx} the \textbf{total bundle} for the given connection.  Unlike $\underline{T}E$ and $T\underline{E}$, the total bundle is a differential bundle over $M$, and its underlying object of $\X \slash M$ is the composite $\widehat{q} = p_Eq:TE \rightarrow M$.  We call $\widehat{\sigma}$ the \textbf{total addition} on $TE$, and we call $\widehat{\zeta}$ and $\widehat{\lambda}$ the \textbf{total zero} and \textbf{total lift} on $TE$, respectively.  
\end{DefSub}

By \ref{thm:bipr_cx}, an affine connection on $M$ induces a `coordinatization' of the second tangent bundle $T^2M$ as a three-fold biproduct $\underline{T}M \oplus_M \underline{T}M \oplus_M \underline{T}M$ over $M$.  The following example illustrates how this coordinatization generalizes the usual coordinate description of the second tangent bundle of a finite-dimensional vector space \pref{exa:sman}:

\begin{ExaSub}[\textbf{The total bundle for the canonical connection}]\label{exa:tb_ccx}
In Example \ref{exa:can_aff_cx}, we considered the case where $\X = \Mf$, $M = V = \RR^n$, and $\underline{E} = \underline{T}V$, and we examined a connection $(K,H)$ on $\underline{T}V = (TV,p_V,+_V,0_V,\ell_V)$, namely the \textit{canonical affine connection on $V$}.  Here $TV = V^2$, and $p_V = \pi_1:V^2 \rightarrow V$ \pref{exa:sman}.  By \ref{thm:bipr_cx}, the connection $(K,H)$ equips $T^2V = V^4$ with the structure of a biproduct of differential bundles
\begin{equation}\label{eq:bipr_ccx}\underline{T}V \oplus_V \underline{T}V \oplus_V \underline{T}V = (T^2V,\widehat{p_V},\widehat{+_V},\widehat{0_V},\widehat{\ell_V})\end{equation}
over $V$ with projection morphisms\footnote{\label{fn:pr_ord_cx}Recall that in \ref{thm:bipr_cx} we chose a specific order for the three projections in question (Footnote \ref{fn:cx_ord}).  In fact, we chose this order so that in the present example the successive projections capture the top-three coordinates $\pi_2,\pi_3,\pi_4$ of $T^2V = V^4$ in ascending order.  Per our convention in \ref{exa:sman}, the highest-order component of $T^2V = V^4$ appears in the rightmost factor, and so we have adopted a similar convention for the biproduct \eqref{eq:bipr_ccx}.}
\begin{equation}\label{eq:exa_pr}
\begin{array}{rcccl}
\pi_1 = p_{TV} = \langle \pi_1,\pi_2\rangle & \;:\; & T^2V = V^4 & \longrightarrow & V^2 = TV\\
\pi_2 = T(p_V) = \langle \pi_1,\pi_3\rangle & \;:\; & T^2V = V^4 & \longrightarrow & V^2 = TV\\
\pi_3 = K = \langle \pi_1,\pi_4\rangle & \;:\; & T^2V = V^4 & \longrightarrow & V^2 = TV,
\end{array}
\end{equation}
for which the given explicit formulae were developed in \ref{exa:sman} and \ref{exa:can_aff_cx}.
The total addition is the map
$$\widehat{+_V}\;:\;V^4 \times_V V^4 \rightarrow V^4$$
that sends a pair of 4-tuples of the form $(x,t_1,u_1,v_1)$, $(x,t_2,u_2,v_2)$ to the 4-tuple $(x,t_1+t_2,u_1+u_2,v_1+v_2)$.  The total zero $\widehat{0_V}:V \rightarrow V^4$ is given by $x \mapsto (x,0,0,0)$.

In order to characterize the total lift $\widehat{\ell_V}$, let us first describe the maps 
$$T(p_{TV}),T^2(p_V),T(K)\;:\;T^3V \rightarrow T^2V$$
obtained by applying $T$ to the projections in \eqref{eq:exa_pr}.  Concretely, we can form the tangent bundle $T^3V = T(V^4)$ of $T^2V = V^4$ as the binary product $T^3V = V^4 \times V^4$, with its first projection to $V^4$ as $p_{T^2V}$.  Since $T^2V = V^4$ and $TV = V^2$ are finite dimensional $\RR$-vector spaces, and the maps in \eqref{eq:exa_pr} are $\RR$-linear, it follows that
$$
\begin{array}{rcccl}
T(p_{TV}) = \langle \pi_1,\pi_2\rangle \times \langle \pi_1,\pi_2\rangle & \;:\; & V^4 \times V^4 & \longrightarrow & V^2 \times V^2 = V^4\\
T^2(p_V) = \langle \pi_1,\pi_3\rangle \times \langle \pi_1,\pi_3\rangle  & \;:\; & V^4 \times V^4 & \longrightarrow & V^2 \times V^2 = V^4\\
T(K) = \langle \pi_1,\pi_4\rangle \times \langle \pi_1,\pi_4\rangle & \;:\; & V^4 \times V^4 & \longrightarrow & V^2 \times V^2 = V^4.
\end{array}
$$
Further, since the projections 
$$p_{TV},T(p_V),K\;\;\;:\;\;(T^2V,\widehat{p_V},\widehat{+_V},\widehat{0_V},\widehat{\ell_V}) \longrightarrow (TV,p_V,+_V,0_V,\ell_V)$$
are linear morphisms of differential bundles, we know that the diagrams
$$
\xymatrix{
V^4 \ar[d]_{p_{TV}} \ar[r]^(.4){\widehat{\ell_V}} & V^4 \times V^4 \ar[d]_{T(p_{TV})} & V^4 \ar[d]_{T(p_V)} \ar[r]^(.4){\widehat{\ell_V}} & V^4 \times V^4 \ar[d]^{T^2(p_V)} & V^4 \ar[d]_{K} \ar[r]^(.4){\widehat{\ell_V}} & V^4 \times V^4 \ar[d]^{T(K)}\\
V^2 \ar[r]_{\ell_V}                           & V^4                               & V^2 \ar[r]_{\ell_V}                           & V^4                              & V^2 \ar[r]_{\ell_V}                           & V^4                              
}
$$
commute.  Using these facts, it is now straightforward to show that the total lift is given by
$$\widehat{\ell_V}\;:\;V^4 \rightarrow V^4 \times V^4,\;\;\;\;\;(x,t,u,v) \mapsto ((x,0,0,0),(0,t,u,v))\;.$$
\end{ExaSub}

\section{The induced horizontal connection}

In the present section, we show that for any effective vertical connection $K$ there is a unique horizontal connection $H$ such that $(K,H)$ is a connection.  Again let $\underline{E}$ be a differential bundle over $M$, satisfying Basic Condition \ref{par:std_assn}.  By Theorem \ref{thm:evcx_bipr}, we know that any effective vertical connection $K$ on $\underline{E}$ endows $TE$ with the structure of a biproduct of differential bundles
\begin{equation}\label{eq:bipr_evcx}\underline{E} \oplus_M \underline{T}M \oplus_M \underline{E},\end{equation}
and we will now show that the canonical injection of $\underline{E} \oplus_M \underline{T}M$ into the latter biproduct is a horizontal connection $H$ on $\underline{E}$.  To this end, we will employ the biproduct decomposition \eqref{eq:bipr_evcx} in order to re-interpret the axioms for a horizontal connection $H$ (\ref{def:cx}(2)), and also the compatibility axioms between $K$ and $H$ (\ref{def:cx}(3)).

Our analysis will be rendered more tractable by working with three arbitrary differential bundles $\underline{F_1}$, $\underline{F_2}$, $\underline{F_3}$ over $M$, but the reader is urged to keep in mind the case where these bundles are $\underline{E}$, $\underline{T}M$, $\underline{E}$, respectively.

\begin{ParSub}[\textbf{Notation}]\label{par:three_db}
Let $\underline{F_1}$, $\underline{F_2}$, $\underline{F_3}$ be differential bundles over $M$, and write $\underline{F_i} = (F_i,r_i,\sigma_i,\zeta_i,\lambda_i)$, $i = 1,2,3$.  Assume that the set of differential bundles $\{\underline{F_1},\underline{F_2},\underline{F_3}\}$ has finite products over $M$ \pref{def:set_has_fp}.  Let 
$$\underline{F} = \underline{F_1} \oplus_M \underline{F_2} \oplus_M \underline{F_3}$$
be a biproduct of differential bundles over $M$, and let $(F,r)$ denote the object of $\X \slash M$ underlying $\underline{F}$, so that $F$ is a fibre product $F_1 \times_M F_2 \times_M F_3$ in $\X$ \pref{thm:dbun_pr_fp}.  Given any enumeration $i,j,k$ of all the elements of $\{1,2,3\}$, we write
$$\iota_{ij}\;:\;\underline{F_i} \oplus_M \underline{F_j} \longrightarrow \underline{F_1} \oplus_M \underline{F_2} \oplus_M \underline{F_3}$$
to denote the canonical \textit{injection}, defined by the equations $\iota_{ij}\pi_i = \pi_1$, $\iota_{ij}\pi_j = \pi_2$, and $\iota_{ij}\pi_k = 0$, where $0$ denotes the relevant zero morphism in the additive category $\DBun_M$.  Also, let us write
$$\pi_{ij} = \langle \pi_i,\pi_j \rangle\;:\;\underline{F_1} \oplus_M \underline{F_2} \oplus_M \underline{F_3} \longrightarrow \underline{F_i} \oplus_M \underline{F_j}$$
to denote the canonical \textit{projection}.

Whereas the object $F$ of $\X$ underlies the differential bundle $\underline{F}$ over $M$, we deduce by \ref{def:jth_pb} that $F$ also carries the structure of a differential bundle over $F_j$ for each $j = 1,2,3$, namely the $j$-th partial bundle of the biproduct $\underline{F} = \underline{F_1} \oplus_M \underline{F_2} \oplus_M \underline{F_3}$.  Writing the elements of the complement $\{1,2,3\}\backslash\{j\}$ as $i,k$, we can express the $j$-th partial bundle of the latter biproduct as
$$r_j^*(\underline{F_i} \oplus_M \underline{F_k}) = (F,\pi_j,\sigma^j,\zeta^j,\lambda^j),$$
thus employing the notational conventions of \S\ref{sec:pbun}.
\end{ParSub}

\begin{LemSub}\label{thm:inj_lem}
Suppose we are given data as in \ref{par:three_db}.  Then the injection 
$$\iota_{12}\;:\;\underline{F_1} \oplus_M \underline{F_2} \longrightarrow \underline{F_1} \oplus_M \underline{F_2} \oplus_M \underline{F_3}$$
is the unique morphism $H:F_1 \times_M F_2 \rightarrow F$ in $\X$ satisfying the following conditions:
\begin{enumerate}
\item $H:r_1^*(\underline{F_2}) \rightarrow r_1^*(\underline{F_2} \oplus_M \underline{F_3})$ is a linear morphism of differential bundles over $F_1$, where $r_1^*(\underline{F_2})$ and $r_1^*(\underline{F_2} \oplus_M \underline{F_3})$ denote the first partial bundles of the biproducts $\underline{F_1} \oplus_M \underline{F_2}$ and $\underline{F} = \underline{F_1} \oplus_M \underline{F_2} \oplus_M \underline{F_3}$, respectively.
\item $H:r_2^*(\underline{F_1}) \rightarrow r_2^*(\underline{F_1} \oplus_M \underline{F_3})$ is a linear morphism of differential bundles over $F_2$, where $r_2^*(\underline{F_1})$ and $r_2^*(\underline{F_1} \oplus_M \underline{F_3})$ denote the second partial bundles of the biproducts $\underline{F_1} \oplus_M \underline{F_2}$ and $\underline{F} = \underline{F_1} \oplus_M \underline{F_2} \oplus_M \underline{F_3}$, respectively.
\item The following diagram commutes
\begin{equation}\label{eq:pi3_diagr}
\xymatrix{
F_1 \times_M F_2 \ar[d]_{\textnormal{proj}} \ar[r]^(.6)H & F \ar[d]^{\pi_3}\\
M \ar[r]_{\zeta_3}        & F_3
}
\end{equation}
where $\textnormal{proj}$ denotes the projection.
\end{enumerate}
\end{LemSub}
\begin{proof}
For the needed uniqueness, suppose that $H$ satisfies the given conditions.  Then conditions 1 and 2 entail that the diagram
$$
\xymatrix{
F_1 \times_M F_2 \ar[dr]_{\pi_i} \ar[r]^(.6)H & F \ar[d]^{\pi_i}\\
                                         & F_i
}
$$
commutes for $i = 1,2$.  Also, the counterclockwise composite in \eqref{eq:pi3_diagr} is the zero morphism $0:\underline{F_1} \oplus_M \underline{F_2} \rightarrow \underline{F_3}$ in $\DBun_M$, so since the projections $\pi_1,\pi_2,\pi_3$ are jointly monic in $\X$ \pref{par:fi_pr} we deduce that $H = \iota_{12}$.

Next, letting $H = \iota_{12}$, we show that $H$ satisfies the given conditions.  Let 
$$\kappa = \langle 1,0\rangle:\underline{F_2} \rightarrow \underline{F_2} \oplus_M \underline{F_3}$$
denote the injection morphism in $\DBun_M$.  By \ref{thm:pb_dbun}, there are cartesian arrows 
$$r_1^*(\underline{F_2}) \longrightarrow \underline{F_2},\;\;\;\;\;\;\;\;r_1^*(\underline{F_2} \oplus_M \underline{F_3}) \longrightarrow \underline{F_2} \oplus_M \underline{F_3}$$
over $r_1:F_1 \rightarrow M$ in $\DBun$, so the injection $\kappa$ induces a morphism
$$r_1^*(\kappa)\;:\;r_1^*(\underline{F_2}) \longrightarrow r_1^*(\underline{F_2} \oplus_M \underline{F_3})$$
in $\DBun_{F_1}$.  Explicitly, the morphism in $\X$ underlying $r_1^*(\kappa)$ is
$$F_1 \times_M \kappa\;:\;F_1 \times_M F_2 \rightarrow F_1 \times_M F_2 \times_M F_3\;.$$
Hence $r_1^*(\kappa) = H$, so 1 holds.  A similar argument shows that 2 holds.  Further, 3 clearly holds.
\end{proof}

\begin{ThmSub}\label{thm:hcx_for_evcx}
Let $K:TE \rightarrow E$ be an effective vertical connection on $\underline{E} = (E,q,\sigma,\zeta,\lambda)$.  Then $TE$ underlies a biproduct $\underline{E} \oplus_M \underline{T}M \oplus_M \underline{E}$ whose associated injection
$$\iota_{12}\;:\;\underline{E} \oplus_M \underline{T}M \longrightarrow \underline{E} \oplus_M \underline{T}M \oplus_M \underline{E}$$
is a horizontal connection
$$H\;:\;E \times_M TM \longrightarrow TE$$
on $\underline{E}$.  Further, $H$ is the unique horizontal connection on $\underline{E}$ such that $(K,H)$ satisfies axiom \textnormal{\ref{def:cx}(C.5)}.
\end{ThmSub}
\begin{proof}
By \ref{thm:evcx_bipr} we know that $TE$ underlies a biproduct $\underline{E} \oplus_M \underline{T}M \oplus_M \underline{E}$ whose first and second partial bundles are $\underline{T}E$ and $T\underline{E}$, respectively, and whose projections are $\pi_1 = p_E$, $\pi_2 = T(q)$, $\pi_3 = K$.  Letting $\underline{F_1} = \underline{E}$, $\underline{F_2} = \underline{T}M$, $\underline{F_3} = \underline{E}$, and employing the notation of \ref{par:three_db} and \ref{thm:inj_lem}, we find that the partial bundles appearing in Lemma \ref{thm:inj_lem} are
$$r_1^*(\underline{F_2}) = q^*(\underline{T}M),\;\;\;\;r_1^*(\underline{F_2} \oplus_M \underline{F_3}) = \underline{T}E$$
$$r_2^*(\underline{F_1}) = p_M^*(\underline{E}),\;\;\;\;r_2^*(\underline{F_1} \oplus_M \underline{F_3}) = T\underline{E}\;.$$
Therefore, given an arbitrary morphism $H:E \times_M TM \rightarrow TE$ in $\X$, we find that conditions 1, 2, 3 in \ref{thm:inj_lem} are equivalent to axioms C.3, C.4, C.5 of \ref{def:cx}, respectively.  Hence we deduce by \ref{thm:inj_lem} that the injection $\iota_{12}$ is the unique morphism $H$ satisfying axioms C.3, C.4, C.5.  This morphism $H = \iota_{12}$ is clearly a section of the projection morphism $\pi_{12} = \langle p_E,T(q)\rangle:TE \rightarrow E \times_M TM$ and hence is a horizontal connection on $\underline{E}$.
\end{proof}

Our next objective is to show that the horizontal connection $H$ that we obtained in \ref{thm:hcx_for_evcx} is \textit{compatible} with the given effective vertical connection $K$, in the sense that the pair $(K,H)$ is a connection.  The challenge in this regard will be to address the `additive' axiom \ref{def:cx}(C.6), and so we now return to the abstract context of \ref{par:three_db} with the aim of reinterpreting this axiom.

\begin{ParSub}\label{par:a_pb}
Again suppose we are given data as in \ref{par:three_db}.  Let $j \in \{1,2,3\}$, and write the elements of the complement $\{1,2,3\} \slash \{j\}$ as $i,k$.  The endomorphisms $\pi_i\iota_i,\pi_k\iota_k:\underline{F} \rightarrow \underline{F}$ associated to the biproduct $\underline{F} = \underline{F_1} \oplus_M \underline{F_2} \oplus_M \underline{F_3}$ satisfy the equations
\begin{equation}\label{eq:endos_pij}\pi_i\iota_i\pi_j = 0 = \pi_k\iota_k\pi_j\;:\;\underline{F} \rightarrow \underline{F_j}\;.\end{equation}
Hence by \ref{par:add_jpb} we can consider their sum
$$\pi_i\iota_i +^j \pi_k\iota_k\;:\;F \rightarrow F$$
under the operation $+^j$ of addition with respect to the $j$-th partial bundle of the above biproduct.  We deduce the following:
\end{ParSub}

\begin{LemSub}\label{thm:inj_add}
In the situation of \ref{par:a_pb}, $\pi_i\iota_i +^j \pi_k\iota_k$ is precisely the composite
$$\xymatrix{\underline{F_1} \oplus_M \underline{F_2} \oplus_M \underline{F_3} \ar[r]^(.6){\pi_{ik}} & \underline{F_i} \oplus_M \underline{F_k} \ar[r]^(.4){\iota_{ik}} & \underline{F_1} \oplus_M \underline{F_2} \oplus_M \underline{F_3}}$$
of the projection $\pi_{ik}$ and the injection $\iota_{ik}$ defined in \ref{par:three_db}.
\end{LemSub}
\begin{proof}
Let $s = \pi_i\iota_i +^j \pi_k\iota_k$ and $t = \pi_{ik}\iota_{ik}$, and let $s_n = s\pi_n$ and $t_n = t\pi_n:F \rightarrow F_n$ for each $n = 1,2,3$.  By \ref{thm:addn_pb}, we deduce that
$$s_i = \pi_i\iota_i\pi_i + \pi_k\iota_k\pi_i = \pi_i + 0 = \pi_i = t_i\;:\;\underline{F} \rightarrow \underline{F_i},$$
where we write $+,0$ for the addition operations and zero morphisms carried by the additive category $\DBun_M$.  Similarly we deduce that $s_k = \pi_k = t_k$.  By \ref{thm:addn_pb} and \eqref{eq:endos_pij} we deduce also that $s_j = 0 = t_j$.  Since $\pi_1,\pi_2,\pi_3$ are jointly monic in $\X$ \pref{par:fi_pr}, the result is proved.
\end{proof}

\begin{ParSub}\label{par:b_pb}
Again suppose we are given data as in \ref{par:three_db}, and let $i,j,k$ be some enumeration of all the elements of $\{1,2,3\}$.  Employing the notation of \ref{par:three_db}, the endomorphisms $\pi_{ik}\iota_{ik},\;\pi_{ij}\iota_{ij}:\underline{F} \rightarrow \underline{F}$ of the biproduct $\underline{F} = \underline{F_1} \oplus_M \underline{F_2} \oplus_M \underline{F_3}$ satisfy the equations
\begin{equation}\label{eq:ieq}\pi_{ik}\iota_{ik}\pi_i = \pi_i = \pi_{ij}\iota_{ij}\pi_i\;\;:\;\;\underline{F} \rightarrow \underline{F_i}.\end{equation}
Hence by \ref{par:add_jpb} we can consider their sum
$$\pi_{ik}\iota_{ik} +^i \pi_{ij}\iota_{ij}\;:\;F \rightarrow F$$
under the operation of addition with respect to the $i$-th partial bundle of the given biproduct.  We deduce the following:
\end{ParSub}

\begin{LemSub}\label{thm:b_pb}
In the situation of \ref{par:b_pb},
$$\pi_{ik}\iota_{ik} +^i \pi_{ij}\iota_{ij} = 1_F\;.$$
\end{LemSub}
\begin{proof}
Let $s = \pi_{ik}\iota_{ik} +^i \pi_{ij}\iota_{ij}$, and let $s_n = s\pi_n$ for each $n \in \{1,2,3\}$.  By \ref{thm:addn_pb} and \eqref{eq:ieq}, we deduce that
$$s_i = \pi_i\;:\;\underline{F} \rightarrow \underline{F_i}$$
$$s_j = \pi_{ik}\iota_{ik}\pi_j + \pi_{ij}\iota_{ij}\pi_j = 0 + \pi_j = \pi_j\;:\;\underline{F} \rightarrow \underline{F_j}$$
$$s_k = \pi_{ik}\iota_{ik}\pi_k + \pi_{ij}\iota_{ij}\pi_k = \pi_k + 0 = \pi_k\;:\;\underline{F} \rightarrow \underline{F_k}$$
where we write $+,0$ for the addition operations and zero morphisms carried by the additive category $\DBun_M$.
\end{proof}

\begin{ThmSub}\label{thm:evcx_det_cx}
Let $K:TE \rightarrow E$ be an effective vertical connection on $\underline{E} = (E,q,\sigma,\zeta,\lambda)$.
\begin{enumerate}
\item There is a unique horizontal connection $H:E \times_M TM \rightarrow TE$ on $\underline{E}$ such that $(K,H)$ is a connection on $\underline{E}$.  Explicitly, $H$ is the injection 
$$\iota_{12}\;:\;\underline{E} \oplus_M \underline{T}M \longrightarrow \underline{E} \oplus_M \underline{T}M \oplus_M \underline{E}$$
considered in \ref{thm:hcx_for_evcx}.
\item Moreover, given any $H:E \times_M TM \rightarrow TE$ in $\X$, the following are equivalent: (a) $(K,H)$ is a connection on $\underline{E}$, (b) $H$ is a horizontal connection on $\underline{E}$ and axiom \textnormal{\ref{def:cx}(C.5)} holds, (c) $H$ is the morphism $\iota_{12}$ considered in $\ref{thm:hcx_for_evcx}$.
\end{enumerate}
\end{ThmSub}
\begin{proof}
It suffices to prove 2.  Clearly (a) implies (b).  By Theorem \ref{thm:hcx_for_evcx}, we know that (b) is equivalent to (c).  Hence it suffices to let $H = \iota_{12}$ and prove that $(K,H)$ is a connection on $\underline{E}$.  Since (c) implies (b) it suffices to show that $(K,H)$ satisfies axiom \textnormal{\ref{def:cx}(C.6)}.  Letting $\underline{F_1} = \underline{E}$, $\underline{F_2} = \underline{T}M$, $\underline{F_3} = \underline{E}$, and $F = TE$, we know by \ref{thm:evcx_bipr} that $F$ underlies a biproduct $\underline{F} = \underline{F_1} \oplus_M \underline{F_2} \oplus_M \underline{F_3}$ with first and second partial bundles $\underline{T}E$ and $T\underline{E}$, respectively, and with projections $\pi_1 = p_E$, $\pi_2 = T(q)$, $\pi_3 = K$ and injections $\iota_1 = 0_E$, $\iota_2 = T(\zeta)$, $\iota_3 = \lambda$.  Taking $i = 1$, $j = 2$, $k = 3$, we find ourselves in the situation of \ref{par:b_pb}, so by Lemma \ref{thm:b_pb} we deduce that
$$\pi_{13}\iota_{13} +^1 \pi_{12}\iota_{12} = 1_F\;:\;F \rightarrow F,$$
but by Lemma \ref{thm:inj_add} we also know that $\pi_{13}\iota_{13} = \pi_1\iota_1 \;+^2\; \pi_3\iota_3\;:\;F \rightarrow F$, so
\begin{equation}\label{eq:gen_csix}(\pi_1\iota_1 +^2 \pi_3\iota_3) +^1 \pi_{12}\iota_{12} = 1_F\;:\;F \rightarrow F.\end{equation}
We claim that the latter equation is precisely axiom \ref{def:cx}(C.6).  To see this, first note that
$$\pi_1\iota_1 = p_E0_E:TE \rightarrow TE,\;\;\pi_3\iota_3 = K\lambda:TE \rightarrow TE,\;\;\iota_{12} = H,$$
and $\pi_{12} = \langle\pi_1,\pi_2\rangle = \langle p_E,T(q) \rangle = U:TE \rightarrow E \times_M TM$ in the notation of \ref{def:cx}, so our equation \eqref{eq:gen_csix} can be rewritten as
$$(K\lambda +^2 p_E0_E) +^1 UH = 1_{TE}\;:\;TE \rightarrow TE$$
since $+^2$ is commutative.  In the terminology of \ref{par:add_jpb}, $+^1$ and $+^2$ denote the operations of addition with respect to the first and second partial bundles $\underline{T}E$ and $T\underline{E}$, respectively, of the biproduct $\underline{F} = \underline{E} \oplus_M \underline{T}M \oplus_M \underline{E}$.  Hence, since the addition carried by $T\underline{E}$ is $T(\sigma)$ we can express $K\lambda +^2 p_E0_E$ as a composite
$$K\lambda +^2 p_E0_E  = \left(TE \xrightarrow{\langle K\lambda,p_E0_E\rangle} TE \times_{TM} TE = T(E \times_M E) \xrightarrow{T(\sigma)} TE\right).$$
We can rewrite the latter composite in terms of the morphism $\mu$ defined in \ref{par:dbun}(4), thus obtaining the following equation
$$K\lambda +^2 p_E0_E = \left(TE \xrightarrow{\langle K,p_E\rangle} E \times_M E \xrightarrow{\mu} TE\right),$$
noting that $\langle K,p_E \rangle$ is well-defined since $Kq = T(q)p_M = p_Eq$ by \ref{def:cx}(C.1) and the naturality of $p$.  Therefore our equation \eqref{eq:gen_csix} can be rewritten as
$$\langle K,p_E\rangle\mu \;+^1\; UH \;=\; 1_{TE}\;:\;TE \rightarrow TE\;.$$
Since $+^1$ is induced by the addition $+_E$ carried by the first partial bundle $\underline{T}E$, this equation is precisely axiom \ref{def:cx}(C.6).
\end{proof}

\begin{RemSub}
Given a vertical connection $K$ and a horizontal connection $H$, on $\underline{E}$, recall that $(K,H)$ is a connection if and only if the compatibility axioms (C.5) and (C.6) of \ref{def:cx} hold.  But by \ref{thm:evcx_det_cx}(2) and \ref{thm:cx_fp}, $(K,H)$ is a connection if and only if (C.5) holds and $K$ is effective.  Thus one obtains an equivalent definition of the notion of connection, in which the compatibility axiom (C.6) is replaced by a condition that does not make reference to $H$.
\end{RemSub}

\section{Connections: Equivalent formulations}\label{sec:cx_eq_form}

Let $\underline{E} = (E,q,\sigma,\zeta,\lambda)$ be a differential bundle over $M$, satisfying Basic Condition \ref{par:std_assn}.  In the present section, we establish certain equivalent formulations of the notion of connection in terms of a single morphism $K:TE \rightarrow E$.

\newpage

\begin{ThmSub}\label{thm:eq_cond_k}
Let $K:TE \rightarrow E$ be a morphism in $\X$.  Then the following conditions are equivalent:
\begin{enumerate}
\item There exists a (necessarily unique) morphism $H:E \times_M TM \rightarrow TE$ in $\X$ such that $(K,H)$ is a connection on $\underline{E}$.
\item $K$ is an effective vertical connection on $\underline{E}$.
\item $TE$ underlies a biproduct of differential bundles $\underline{E} \oplus_M \underline{T}M \oplus_M \underline{E}$ whose third projection is $K$, whose third injection is $\lambda$, and whose first and second partial bundles are $\underline{T}E$ and $T\underline{E}$, respectively.
\item $K$ is a retraction of $\lambda$, and $TE$ underlies a product of differential bundles \newline $\underline{E} \times_M \underline{T}M \times_M \underline{E}$ whose third projection is $K$ and whose first and second partial bundles are $\underline{T}E$ and $T\underline{E}$, respectively.
\end{enumerate}
\end{ThmSub}
\begin{proof}
A morphism $H$ as in 1 is unique if it exists, by \ref{par:uniq_hc_vc} (or, alternatively, by \ref{thm:evcx_det_cx} and \ref{thm:cx_fp}, which were proved without recourse to \ref{par:uniq_hc_vc}).

By \ref{thm:cx_fp}, 1 implies 2.  By Theorem \ref{thm:evcx_bipr}, 2 implies 3.  Also, 3 clearly implies 4.  Further, 2 implies 1, by Theorem \ref{thm:evcx_det_cx}.  Hence it suffices to prove that 4 implies 2.

Suppose that 4 holds.   Since the first and second partial bundles of the product $\underline{E} \times_M \underline{T}M \times_M \underline{E}$ are $\underline{T}E$ and $T\underline{E}$, respectively, this product necessarily has projections $\pi_1 = p_E$, $\pi_2 = T(q)$, and $\pi_3 = K$.  Hence, since $K$ is a retraction of $\lambda$ and finite products in $\DBun_M$ are biproducts \eqref{thm:dbunm_add}, we deduce by Lemma \ref{thm:vcx_bipr2} that $K$ is a vertical connection on $\underline{E}$.  By \ref{thm:vcx_bipr1} we therefore deduce that 2 holds.
\end{proof}

We now deduce that a connection is equivalently given by a morphism $K$ satisfying the equivalent conditions in \ref{thm:eq_cond_k}, with no further specified structure:

\begin{ThmSub}\label{thm:main}
Let $\underline{E} = (E,q,\sigma,\zeta,\lambda)$ be a differential bundle over an object $M$ in an arbitrary tangent category $\X$, and assume that $\underline{E}$ satisfies Basic Condition \ref{par:std_assn}.
\begin{enumerate}
\item A connection on $\underline{E}$ is equivalently given by a morphism $K:TE \rightarrow E$ such that $TE$ underlies a biproduct of differential bundles $\underline{E} \oplus_M \underline{T}M \oplus_M \underline{E}$ whose third projection is $K$, whose third injection is $\lambda$, and whose first and second partial bundles are $\underline{T}E$ and $T\underline{E}$, respectively.
\item A connection on $\underline{E}$ is equivalently given by a retraction $K:TE \rightarrow E$ of $\lambda$ such that $TE$ underlies a product of differential bundles $\underline{E} \times_M \underline{T}M \times_M \underline{E}$ whose third projection is $K$ and whose first and second partial bundles are $\underline{T}E$ and $T\underline{E}$, respectively.
\item A connection on $\underline{E}$ is equivalently given by an effective vertical connection $K$ on $\underline{E}$, i.e., a vertical connection $K:TE \rightarrow E$ such that the following is a fibre product diagram in $\X$:
$$
\xymatrix{
            & TE \ar[dl]_{p_E} \ar[d]|{T(q)} \ar[dr]^K & \\
E \ar[dr]_q & TM \ar[d]|{p_M}                          & E \ar[dl]^q\\
   & M                                                 &
}
$$
\end{enumerate}
\end{ThmSub}
\begin{proof}
By Theorem \ref{thm:eq_cond_k}, the assignment $(K,H) \mapsto K$ defines a bijective correspondence between connections $(K,H)$ on $\underline{E}$ and morphisms $K:TE \rightarrow E$ satisfying the equivalent conditions of Theorem \ref{thm:eq_cond_k}.
\end{proof}

\begin{RemSub}
Although Theorem \ref{thm:main}(1) demands the \textit{existence} of a biproduct with certain further properties, such a biproduct is in fact unique up to \textit{concrete isomorphism} (\ref{eq:ci}, \ref{rem:kbp_un_ci}), so its underlying object of $\X \slash M$, its zero section, and its lift, as well as its projections and injections, are uniquely determined (\ref{eq:ci}, \ref{rem:kbp_un_ci}).  Indeed, if $\underline{F} = \underline{E} \oplus_M \underline{T}M \oplus_M \underline{E}$ is a biproduct with the given properties, then its projection morphisms must be $\pi_1 = p_E$, $\pi_2 = T(q)$, $\pi_3 = K$, so its underlying object of $\X \slash M$ must be $\widehat{q} = p_Eq:TE \rightarrow M$ and its injection morphisms must be $\iota_1 = 0_E$, $\iota_2 = T(\zeta)$, $\iota_3 = \lambda$, by \ref{thm:par_bun}(3).
\end{RemSub}

Theorem \ref{thm:main} entails the following characterizations of affine connections, recalling from \ref{par:exa_std_assn} that the tangent bundle $\underline{T}M$ always satisfies Basic Condition \ref{par:std_assn}:

\begin{ThmSub}
Let $M$ be an object of an arbitrary tangent category $\X$.
\begin{enumerate}
\item An affine connection on $M$ is equivalently given by a morphism $K:T^2M \rightarrow TM$ such that $T^2M$ underlies a biproduct of differential bundles $\underline{T}M \oplus_M \underline{T}M \oplus_M \underline{T}M$ whose third projection is $K$, whose third injection is $\ell_M:TM \rightarrow T^2M$, and whose first and second partial bundles are $\underline{T}(TM)$ and $T(\underline{T}M)$, respectively.
\item An affine connection on $M$ is equivalently given by a retraction $K:T^2M \rightarrow TM$ of $\ell_M:TM \rightarrow T^2M$ such that $T^2M$ underlies a product of differential bundles $\underline{T}M \times_M \underline{T}M \times_M \underline{T}M$ whose third projection is $K$ and whose first and second partial bundles are $\underline{T}(TM)$ and $T(\underline{T}M)$, respectively.
\item An affine connection on $M$ is equivalently given by an effective vertical connection $K$ on $\underline{T}M$, i.e., a vertical connection $K:T^2M \rightarrow TM$ on $\underline{T}M$ such that the following is a fibre product diagram in $\X$:
$$
\xymatrix{
                 & T^2M \ar[dl]_{p_{TM}} \ar[d]|{T(p_M)} \ar[dr]^K & \\
TM \ar[dr]_{p_M} & TM \ar[d]|{p_M}                                 & TM \ar[dl]^{p_M}\\
                 & M                                              &
}
$$
\end{enumerate}
\end{ThmSub}

\bibliographystyle{amsplain}
\bibliography{bib}

\newcommand{\noopsort}[1]{}
\providecommand{\bysame}{\leavevmode\hbox to3em{\hrulefill}\thinspace}
\providecommand{\MR}{\relax\ifhmode\unskip\space\fi MR }
\providecommand{\MRhref}[2]{%
  \href{http://www.ams.org/mathscinet-getitem?mr=#1}{#2}
}
\providecommand{\href}[2]{#2}
\begin{thebibliography}{1}

\bibitem{BCL}
R.~F. Blute, G.~S.~H. Cruttwell, and R.~B.~B. Lucyshyn-Wright, \emph{Affine
  geometric spaces in tangent categories}, (preprint),
  \href{http://arxiv.org/abs/1807.09554}{arXiv:1807.09554}, 2018.

\bibitem{CoCr:TCats}
J.~R.~B. Cockett and G.~S.~H. Cruttwell, \emph{Differential structure, tangent
  structure, and {SDG}}, Appl. Categ. Structures \textbf{22} (2014), no.~2,
  331--417.

\bibitem{CoCr:Conn}
\bysame, \emph{Connections in tangent categories}, Theory Appl. Categ.
  \textbf{32} (2017), no. 26, 835--888.

\bibitem{CoCr:DBun}
\bysame, \emph{Differential bundles and fibrations for tangent categories},
  Cah. Topol. G\'eom. Diff\'er. Cat\'eg. \textbf{LIX} (2018), 10--92.

\bibitem{Dom}
P.~Dombrowski, \emph{On the geometry of the tangent bundle}, J. Reine Angew.
  Math. \textbf{210} (1962), 73--88.

\bibitem{Kock:HiCx}
A.~Kock, \emph{Infinitesimal cubical structure, and higher connections},
  arXiv:0705.4406, 2007.

\bibitem{LaNi}
R.~Lavendhomme and H.~Nishimura, \emph{Differential forms in synthetic
  differential geometry}, Internat. J. Theoret. Phys. \textbf{37} (1998),
  no.~11, 2823--2832.

\bibitem{Pat:Cx}
L.-N. Patterson, \emph{Connexions and prolongations}, Canad. J. Math.
  \textbf{27} (1975), no.~4, 766--791.

\bibitem{Ros:Tan}
J.~Rosick{\'y}, \emph{Abstract tangent functors}, Diagrammes \textbf{12}
  (1984), JR1--JR11.

\end{thebibliography}

\end{document}